\documentclass{article}
\usepackage[latin9]{inputenc}
\usepackage{mathtools}
\usepackage{amsmath}
\usepackage{amsthm}
\usepackage{amssymb}
\usepackage{esint}

\makeatletter

\providecommand{\tabularnewline}{\\}

\theoremstyle{plain}
\newtheorem{thm}{\protect\theoremname}
\theoremstyle{remark}
\newtheorem{rem}[thm]{\protect\remarkname}
\theoremstyle{definition}
\newtheorem{example}[thm]{\protect\examplename}
\theoremstyle{plain}
\newtheorem{prop}[thm]{\protect\propositionname}
\theoremstyle{remark}
\newtheorem*{rem*}{\protect\remarkname}
\theoremstyle{definition}
\newtheorem{defn}[thm]{\protect\definitionname}

\makeatother

\providecommand{\definitionname}{Definition}
\providecommand{\examplename}{Example}
\providecommand{\propositionname}{Proposition}
\providecommand{\remarkname}{Remark}
\providecommand{\theoremname}{Theorem}

\begin{document}
\title{The universal logic of repeated experiments}
\author{Sergio Grillo\\
{\small{}Instituto Balseiro, Universidad Nacional de Cuyo and CONICET}\\[-5pt]
{\small{}Av. Bustillo 9500, San Carlos de Bariloche}\\[-5pt] {\small{}R8402AGP,
República Argentina}\\[-5pt] {\small{}$sergrill@gmail.com$}}
\maketitle
\begin{abstract}
Let $\mathsf{E}$ be the event space of an experiment that can be
indefinitely repeated. A natural question arises: given a countable
cardinal $\kappa$, which is the event space of the $\kappa$-times
repeated experiment? In the case of classical experiments, where $\mathsf{E}$
is a (complete) Boolean algebra on some set $S$, i.e. a \textit{classical
}or\textit{ distributive logic}, the answer is more or less well known:
the (complete) Boolean algebra on $S^{\kappa}$ generated by $\mathsf{E}^{\kappa}$.
But, what if $\mathsf{E}$ is not a Boolean algebra? In this paper
we give a constructive answer to this question for any $\kappa$ and
in the context of general orthocomplemented complete lattices, i.e.
\textit{general logics}. Concretely, given a general logic $\mathsf{E}$
defining the event space of a given experiment, we construct a logic
$\mathsf{U}_{\kappa}\left(\mathsf{E}\right)$ representing the event
space of the $\kappa$-times repeated experiment, in such a way that
$\mathsf{U}_{\kappa}\left(\mathsf{E}\right)$ is distributive if and
only if so is $\mathsf{E}$, and, as expected, $\mathsf{U}_{\kappa}\left(\mathsf{E}\right)$
and $\mathsf{E}$ are isomorphic when $\kappa=1$. We also extend
our construction to the case in which the event space changes from
one repetition to another and the cardinal $\kappa$ is arbitrary.
This gives rise to tensor products $\bigotimes_{\alpha\in\kappa}\mathsf{E}_{\alpha}$
of families $\left\{ \mathsf{E}_{\alpha}\right\} _{\alpha\in\kappa}$
of orthocomplemented complete lattices, in terms of which $\mathsf{U}_{\kappa}\left(\mathsf{E}\right)=\bigotimes_{\alpha\in\kappa}\mathsf{E}$.
\end{abstract}
\tableofcontents{}

\section{Introduction}

Given an experiment with set of results $S$ (the sample space),\footnote{To understand what we extactly mean by ``experiment'' and by ``result'',
see the Appendix.} consider a set $\mathsf{E}$ of propositions about such results:
the \textbf{event space} of the experiment. Assume $\mathsf{E}$ is
closed under the logical connectives: AND, OR and NOT, and contain
the trivial and absurd propositions. For a (physical) experiment involving
a classical system, the considered propositions used to be those of
the form: ``the result belongs to the subset $a\subseteq S$''.
So, $\mathsf{E}$ can be identified with a Boolean algebra on $S$.
In particular, $\mathsf{E}$ is a bounded (ortho)complemented distributive
lattice. The order $\subseteq$ of $\mathsf{E}$ is given by the inclusion
of subsets, the join $\cup$ by their union, the meet $\cap$ by their
intersection and the orthocomplementation $\cdot{}^{c}$ by the complement
of subsets. The bottom is $\mathbf{0}\coloneqq\emptyset$ and the
top is $\mathbf{1}\coloneqq S$. On the other hand, for experiments
involving a quantum system on a Hilbert space $\mathcal{H}$, the
event space $\mathsf{E}$ is identified with the set of all closed
subspaces of $\mathcal{H}$. (This is because the results of the experiment
are the eigenvalues of the self-adjoint operators on $\mathcal{H}$,
and each closed subspace is the eigenspace of some self-adjoint operator).
In this case, the event space is a complete orthocomplemented lattice
$\left(\mathsf{E},\subseteq,\cup,\cap,^{c},\mathbf{0},\mathbf{1}\right)$,
but it is not distributive. The order is given by the inclusion of
subspaces, the join by the closure of the sum, the meet by the intersection
and the orthocomplementation by the orthogonal complement. The bottom
and the top elements are the trivial subspace and the entire space,
respectively. In any case, the dictionary between lattice operations
and the logical connectives among propositions is given as follows:

\bigskip{}

\begin{tabular}{|c|c|}
\hline 
{\footnotesize{}$a\subseteq b$} & {\footnotesize{}$a\textrm{ implies }b$}\tabularnewline
\hline 
{\footnotesize{}$a\cup b$} & {\footnotesize{}$a\textrm{ or }b$}\tabularnewline
\hline 
{\footnotesize{}$a\cap b$} & {\footnotesize{}$a\textrm{ and }b$}\tabularnewline
\hline 
{\footnotesize{}$a^{c}$} & {\footnotesize{}$\textrm{not }a$}\tabularnewline
\hline 
\end{tabular}$\qquad$%
\begin{tabular}{|c|c|}
\hline 
{\footnotesize{}$0$} & {\footnotesize{}$\textrm{absurd proposition}$}\tabularnewline
\hline 
{\footnotesize{}$1$} & {\footnotesize{}$\textrm{trivial proposition}$}\tabularnewline
\hline 
\end{tabular}

\bigskip{}

Suppose we have an experiment that can be repeated an arbitrary number
of times. Then, a natural question arises: given $n\in\mathbb{N}$,
if the event space of the original experiment is $\mathsf{E}$, which
is the event space of the $n$-times repeated experiment? In the case
in which $\mathsf{E}$ is a Boolean algebra, the answer to this question
is more or less well known (although is not easy to find it explicitly
in the literature), and it is described in the Section \ref{cle}
of this paper. But, which is the answer if $\mathsf{E}$ is not a
Boolean algebra? The main aim of the paper is, given a countable cardinal
$\kappa$ and a (not necessarily distributive) complete orthocomplemented
lattice $\mathsf{E}$ (representing the event space of a given repeatable
experiment), to construct a complete orthocomplemented lattice $\mathsf{U}_{\kappa}\left(\mathsf{E}\right)$
representing, in a ``universal'' way, the event space of the $\kappa$-times
repeated experiment (including the case of infinite repetitions).
We say \textit{universal} in the following sense: if for some reason
the event space of the $\kappa$-times repeated experiment is taken
as a given complete orthocomplemented lattice $\mathsf{F}$, then
there must exist a unique epimorphism $\mathsf{U}_{\kappa}\left(\mathsf{E}\right)\twoheadrightarrow\mathsf{F}$.
Roughly speaking, $\mathsf{F}$ must be a quotient of $\mathsf{U}_{\kappa}\left(\mathsf{E}\right)$.
This construction is done along the Sections \ref{gexp} to \ref{unlo}.
In Section \ref{doue}, we show that $\mathsf{U}_{\kappa}\left(\mathsf{E}\right)$
is distributive if and only if so is $\mathsf{E}$. 

Our construction for $\left|\kappa\right|=1$ enable us to define,
for every orthocomplemented lattice $\mathsf{L}$, a complete orthocomplemented
lattice (see Section \ref{eec}), which is isomorphic to $\mathsf{L}$
if and only if $\mathsf{L}$ is complete, and which is distributive
if and only if so is $\mathsf{L}$ (as happen for the Dedekind-MacNeille
completion). 

Also, our construction can be extended to any cardinal $\kappa$ (not
necessarily countable), and even to the case in which the event space
changes from one repetition to another. Concretely, if the involved
event spaces are $\mathsf{E}_{\alpha}$, with $\alpha\in\kappa$,
then our construction gives rise to a particular kind of tensor product
$\bigotimes_{\alpha\in\kappa}\mathsf{E}_{\alpha}$ of orthocomplemented
complete lattices (see Section \ref{tpp}). (A similar object can
be found in Ref. \cite{sik}, p. 175, but in the restricted context
of Boolean algebras). It is worth mentioning that, for this tensor
product, the involved factors are not necessarily completely distributive
(see \cite{kubi}).

Finally, in Section \ref{fw}, we explain how we shall use in a future
work the logic $\mathsf{U}_{\kappa}\left(\mathsf{E}\right)$ to study
the foundations of \textit{subjective} probability, i.e. \textit{reasonable
expectation}, in the same sense as Cox \cite{cox} (see also \cite{cox2,jaynes,paris}),
but on general (non necessarily distributive) event spaces $\mathsf{E}$.

\bigskip{}

We shall assume the reader is familiar with the basic concepts of
lattice theory (see for instance \cite{blyth}).

\section{Classical experiments}

\label{cle} Consider a (classical) random experiment with sample
space $S$ and event space $\mathsf{E}\subseteq2^{S}$. We shall assume
that $\mathsf{E}$ is a \textit{complete Boolean algebra} on $S$,
i.e. $\mathsf{E}$ is closed under set complementation and arbitrary
intersections, and \textit{ipso facto} unions. (Note that, as a consequence,
$\emptyset,S\in\mathsf{E}$). In other words, we shall assume that
$\mathsf{E}$ is what we shall call a \textbf{classical logic}: a
complete (ortho)complemented \textit{distributive} lattice. 
\begin{rem}
For certain experiments, as those related to a statistical classical
mechanical system, the event space $\mathsf{L}$ is usually taken
as a $\sigma$-algebra, which is not necessarily complete. In theses
cases, we can consider the (complemented and distributive) completion
$\mathsf{E}_{\mathsf{L}}$ of $\mathsf{L}$, defined in Section \ref{eec}.
\end{rem}

\subsection{The event space of the repeated experiment}

If the experiment can be repeated indefinitely, it is natural to take
the sample space of the ``$N$-times repeated experiment'' as $S^{N}=S\times\cdots\times S$
(with $N$ factors) and its event space as the subset of $2^{S^{N}}$
formed out by the cartesian products 
\[
a_{1}\times\cdots\times a_{N}\subseteq S^{N},\quad\textrm{with}\quad a_{i}\in\mathsf{E}\quad\textrm{for all\ensuremath{\quad i=1,...,N}},
\]
that is to say, the elements
\[
a_{1}\times\cdots\times a_{N}\in\mathsf{E}^{N}=\underbrace{\mathsf{E}\times\cdots\times\mathsf{E}}_{N\textrm{-times}},
\]
together with all its possible unions inside $S^{N}$. This family
of sets gives exactly the complete Boolean subalgebra of $2^{S^{N}}$
generated by $\mathsf{E}^{N}$, which we shall indicate $\left\langle \mathsf{E}^{N}\right\rangle $,
i.e. 
\begin{equation}
\left\langle \mathsf{E}^{N}\right\rangle \coloneqq\left\{ {\textstyle \bigcup_{i\in I}}A_{i}\;:\;A_{i}\in\mathsf{E}^{N}\right\} .\label{EN}
\end{equation}

\begin{rem}
Above construction is similar to that appearing in the context of
\textit{product measures }(see for instance \cite{ash}, p. 101),
where the elements of the form $a_{1}\times\cdots\times a_{N}\in\mathsf{E}^{N}$
are called \textit{rectangles}.
\end{rem}

$\,$
\begin{rem}
Note that the intersection is closed inside $\mathsf{E}^{N}$. Concretely,
given 
\[
a_{1}\times\cdots\times a_{N},\;b_{1}\times\cdots\times b_{N}\in\mathsf{E}^{N},
\]
we have that 
\[
\left(a_{1}\times\cdots\times a_{N}\right)\cap\left(b_{1}\times\cdots\times b_{N}\right)=\left(a_{1}\cap b_{1}\right)\times\cdots\times\left(a_{N}\cap b_{N}\right).
\]
\end{rem}

$\,$
\begin{rem}
\label{vac} Recall that $a_{1}\times\cdots\times a_{N}=\emptyset$
if $a_{i}=\emptyset$ for some $i$. 
\end{rem}

Each element $a_{1}\times\cdots\times a_{N}\in\mathsf{E}^{N}$ represents
the event of the $N$-times repeated experiment in which $a_{1}$
occurs in the first repetition (of the original experiment), $a_{2}$
occurs in the second repetition, and so on. The union of two elements
\[
A=a_{1}\times\cdots\times a_{N},\quad B=b_{1}\times\cdots\times b_{N}\in\mathsf{E}^{N},
\]
represents the proposition ``$A$ or $B$'' of the $N$-times repeated
experiment. 
\begin{rem}
\label{un} It can be shown that, if $N>1$, such an union does not
belong to $\mathsf{E}^{N}$, unless, for instance, if: $a_{i}\subseteq b_{i}$,
for all $i=1,...,N$, or $b_{i}\subseteq a_{i}$, for all $i=1,...,N$.
In both cases, 
\[
\left(a_{1}\times\cdots\times a_{N}\right)\cup\left(b_{1}\times\cdots\times b_{N}\right)=\left(a_{1}\cup b_{1}\right)\times\cdots\times\left(a_{N}\cup b_{N}\right).
\]
\end{rem}

Analogously, for the ``indefinitely repeated experiment,'' the sample
space would be the set of sequences $S^{\mathbb{N}}$ (that can be
seen as $\infty$-tuples of elements of $S$) and the event space
can be taken as the subset of $2^{S^{\mathbb{N}}}$ formed out by
the cartesian products
\[
a_{1}\times\cdots\times a_{n}\times\cdots\subseteq S^{\mathbb{N}},\quad\textrm{with}\quad a_{n}\in\mathsf{E}\quad\textrm{for all\ensuremath{\quad n\in\mathbb{N}}},
\]
together with all its possible unions (inside $S^{\mathbb{N}}$).
Above cartesian products can also be written as the $\mathsf{E}$-valued
sequences
\[
\left(a_{1},\ldots,a_{n},\ldots\right)\in\mathsf{E}^{\mathbb{N}}.
\]
The event space so defined is precisely the complete Boolean subalgebra
of $2^{S^{\mathbb{N}}}$ generated by $\mathsf{E}^{\mathbb{N}}$,
which we shall indicate $\left\langle \mathsf{E}^{\mathbb{N}}\right\rangle $;
i.e. 
\begin{equation}
\left\langle \mathsf{E}^{\mathbb{N}}\right\rangle \coloneqq\left\{ {\textstyle \bigcup_{i\in I}}A_{i}\;:\;A_{i}\in\mathsf{E}^{\mathbb{N}}\right\} .\label{EBN}
\end{equation}
Of course, the sequence $\left(a_{1},\ldots,a_{n},\ldots\right)\in\mathsf{E}^{\mathbb{N}}$
represents the event of the indefinitely repeated experiment in which
$a_{n}$ occurs in the $n$-th repetition (of the original experiment),
for all $n\in\mathbb{N}$. 

\bigskip{}

Note that, since each $\left\langle \mathsf{E}^{N}\right\rangle $,
with $N\in\mathbb{N}$, as well as $\left\langle \mathsf{E}^{\mathbb{N}}\right\rangle $,
are complete Boolean algebras of sets, then they are completely distributive
lattices. The later means that, for all subsets $\left\{ A_{j}^{i}\;:\;j\in J_{i},\;i\in I\right\} $
of the algebra, the identity
\[
{\textstyle \bigcap_{i\in I}\bigcup_{j\in J_{i}}}A_{j}^{i}={\textstyle \bigcup_{f\in F}\bigcap_{i\in I}}A_{f\left(i\right)}^{i}
\]
(and its dual) holds, where $F$ is the set of (choice) functions
\[
F=\left\{ {\textstyle f:I\rightarrow\bigcup_{i\in I}J_{i}}\;/\;f\left(i\right)\in J_{i}\right\} .
\]

Below, and for latter convenience, we shall study the complementation
of each $\left\langle \mathsf{E}^{N}\right\rangle $ and $\left\langle \mathsf{E}^{\mathbb{N}}\right\rangle $,
in terms of the complementation of $\mathsf{E}$, taking into account
the complete distributivity property.

\subsection{An expression for the complement }

Given an element $a_{1}\times\cdots\times a_{N}\in\mathsf{E}^{N}$,
its complement (as a subset of $S^{N}$) can be written as
\[
{\textstyle \left(a_{1}\times\cdots\times a_{N}\right)^{c}=\bigcup_{k\in\left[N\right]}S\times\cdots\times a_{k}^{c}\times\cdots\times S}.
\]
And given a generic element of $\left\langle \mathsf{E}^{N}\right\rangle $,
say
\[
{\textstyle \bigcup_{i\in I}}a_{1}^{i}\times\cdots\times a_{N}^{i},
\]
with $I$ an index set, since
\[
\left({\textstyle \bigcup_{i\in I}}a_{1}^{i}\times\cdots\times a_{N}^{i}\right)^{c}={\textstyle \bigcap_{i\in I}}\left(a_{1}^{i}\times\cdots\times a_{N}^{i}\right)^{c},
\]
using the complete distributivity we can write 
\begin{equation}
\left({\textstyle \bigcup_{i\in I}}a_{1}^{i}\times\cdots\times a_{N}^{i}\right)^{c}={\textstyle \bigcup_{f\in F}}{\textstyle \bigcap_{i\in I}}S\times\cdots\times\left(a_{f\left(i\right)}^{i}\right)^{c}\times\cdots\times S,\label{fe}
\end{equation}
where 
\[
F=\left\{ f:I\rightarrow\left[N\right]\right\} .
\]
Also, since

\[
\begin{array}{lll}
{\textstyle \bigcap_{i\in I}}S\times\cdots\times\left(a_{f\left(i\right)}^{i}\right)^{c}\times\cdots\times S & = & \bigcap_{k\in\left[N\right]}\bigcap_{i\in f^{-1}\left(k\right)}S\times\cdots\times\left(a_{k}^{i}\right)^{c}\times\cdots\times S\\
\\
 & = & \left(\bigcap_{i\in f^{-1}\left(1\right)}\left(a_{1}^{i}\right)^{c}\right)\times\cdots\times\left(\bigcap_{i\in f^{-1}\left(N\right)}\left(a_{N}^{i}\right)^{c}\right),
\end{array}
\]
we can re-write Eq. \eqref{fe} as
\begin{equation}
{\textstyle \left({\textstyle \bigcup_{i\in I}}a_{1}^{i}\times\cdots\times a_{N}^{i}\right)^{c}={\textstyle \bigcup_{f\in F}}\left(\bigcap_{i\in f^{-1}\left(1\right)}\left(a_{1}^{i}\right)^{c}\right)\times\cdots\times\left(\bigcap_{i\in f^{-1}\left(N\right)}\left(a_{N}^{i}\right)^{c}\right).}\label{fe2}
\end{equation}

\bigskip{}

Analogously, given a sequence $\left(a_{1},\ldots,a_{n},\ldots\right)\in\mathsf{E}^{\mathbb{N}}$,
its complement can be written as
\[
{\textstyle \left(a_{1},\ldots,a_{n},\ldots\right)^{c}=\bigcup_{k\in\mathbb{N}}\left(S,\ldots,S,a_{k}^{c},S,\ldots\right)}.
\]
So, for a generic member 
\[
{\textstyle \bigcup_{i\in I}}\left(a_{1}^{i},\ldots,a_{n}^{i},\ldots\right)\in\left\langle \mathsf{E}^{\mathbb{N}}\right\rangle ,
\]
we have that
\begin{equation}
{\textstyle \left({\textstyle \bigcup_{i\in I}}\left(a_{1}^{i},\ldots,a_{n}^{i},\ldots\right)\right)^{c}={\textstyle \bigcup_{f\in F}}\left(\bigcap_{i\in f^{-1}\left(1\right)}\left(a_{1}^{i}\right)^{c},\ldots,\bigcap_{i\in f^{-1}\left(k\right)}\left(a_{k}^{i}\right)^{c},\ldots\right)},\label{fen}
\end{equation}
where $F=\left\{ f:I\rightarrow\mathbb{N}\right\} $. An equation
similar to \eqref{fen} will appear latter, in a more general context.

\section{General experiments}

\label{gexp} Now, suppose that we have an experiment for which its
event space $\mathsf{E}$ is not necessarily a classical logic, but
a general\textbf{ logic}: a complete orthocomplemented (\textit{non-necessarily
distributive}) lattice. Its order, joint, meet and orthocomplementation
will be denoted: $\subseteq$, $\cup$, $\cap$ and $^{c}$, while
its bottom and top elements as $\mathbf{0}$ and $\mathbf{1}$, respectively.
Note that $\mathbf{0}$ represent the absurd proposition (which is
always false) and $\mathbf{1}$ the trivial proposition (which is
always true).
\begin{example}
\label{eq} For a quantum system on a Hilbert space $\mathcal{H}$,
the event space is taken as $\mathsf{L}\left(\mathcal{H}\right)$:
the set of closed linear sub-spaces of $\mathcal{H}$ together with
the set inclusion. This is a complete orthocomplemented lattice, and
it is not distributive. More precisely, $\mathsf{L}\left(\mathcal{H}\right)$
is a complete (irreducible and atomic) ortho-modular lattice, also
called a \textbf{quantum logic}. (For more details, see for instance
Ref. \cite{kal}). The meet is given by the set intersection, and
the join of a family of subspaces is given by the smaller closed subspace
containing them. As expected, the orthocomplementation is given by
the orthogonal complement, and the bottom and top elements are $\mathbf{0}=\left\{ 0\right\} $
(the trivial subspace) and $\mathbf{1}=\mathcal{H}$, respectively.
Alternatively, $\mathsf{L}\left(\mathcal{H}\right)$ can be thought
of as the set of orthogonal projectors on $\mathcal{H}$. In such
a case, $\mathbf{0}$ is the null operator and $\mathbf{1}$ is the
identity operator.
\end{example}

If the experiment can be indefinitely repeated, we shall study, as
we did above in the classical case, how it should be like the event
space of the ``$N$-times repeated experiment'' and that of the
``indefinitely repeated experiment.'' 

\subsection{The lattices $\mathsf{E}^{\kappa}$}

Let us consider the sets $\mathsf{E}^{\kappa}$, with $\kappa=\left[N\right]=\left\{ 1,...,N\right\} $,
for some $N\in\mathbb{N}$, or $\kappa=\mathbb{N}$. For $\kappa=\left[N\right]$
we have the set of $N$-tuples 
\[
\left(a_{1},\ldots,a_{N}\right)\in\mathsf{E}^{\left[N\right]}=\mathsf{E}^{N}=\underbrace{\mathsf{E}\times\cdots\times\mathsf{E}}_{N\textrm{-times}},
\]
and for $\kappa=\mathbb{N}$ we have the set of sequences
\[
\left(a_{1},\ldots,a_{n},\ldots\right)\in\mathsf{E}^{\mathbb{N}}.
\]
To unify notation, we shall denote $\left(a_{1},\ldots,a_{n},\ldots\right)$
the elements of $\mathsf{E}^{N}$ as well as $\mathsf{E}^{\mathbb{N}}$.
Of course, such elements represent the events (or propositions) of
the ``$\kappa$-times repeated experiment'' in which $a_{n}$ occurs
(or is true) in the $n$-th repetition of the original experiment.

Let us identify among themselves all the elements of $\mathsf{E}^{\kappa}$
with unless some component equal to $\mathbf{0}$. It is clear that
all of them will be identified with
\[
\hat{\mathbf{0}}\coloneqq\begin{cases}
\underbrace{\left(\mathbf{0},\ldots,\mathbf{0}\right)}_{N\textrm{-times}}, & \quad\kappa=\left[N\right],\\
\left(\mathbf{0},\ldots,\mathbf{0},\ldots\right), & \quad\kappa=\mathbb{N}.
\end{cases}
\]
We must do that because, if some component $a_{n}$ of $\left(a_{1},\ldots,a_{n},\ldots\right)\in\mathsf{E}^{\kappa}$
is the absurd proposition of the original experiment (i.e. $a_{n}=\mathbf{0}$),
which is always false, then the proposition represented by $\left(a_{1},\ldots,a_{n},\ldots\right)$
must be always false too, so it must be the absurd proposition of
the $\kappa$-times repeated experiment. This identification is immediate
in the classical case (see Remark \ref{vac}).

It is easy to show that $\mathsf{E}^{\kappa}$, for $\kappa=\left[N\right]$
and $\kappa=\mathbb{N}$, is a complete lattice with respect to the
component-wise order, joint and meet, which we shall also denote $\subseteq$,
$\cup$ and $\cap$. Its bottom is $\hat{\mathbf{0}}$ and its top
is

\[
\hat{\mathbf{1}}\coloneqq\begin{cases}
\underbrace{\left(\mathbf{1},\ldots,\mathbf{1}\right)}_{N\textrm{-times}}, & \quad\kappa=\left[N\right],\\
\left(\mathbf{1},\ldots,\mathbf{1},\ldots\right), & \quad\kappa=\mathbb{N}.
\end{cases}
\]

\subsection{The free terms of $\mathsf{E}^{\kappa}$}

Inspired in Section \ref{cle} (see Eqs. \eqref{EN} and \eqref{EBN}),
consider the set of \textit{free terms}
\begin{equation}
\mathsf{T}\left(\mathsf{E}^{\kappa}\right)\coloneqq\left\{ {\textstyle \bigsqcup_{i\in I}}A^{i}\;:\;\left|I\right|\leq2^{\left|\mathsf{E}^{\kappa}\right|},\quad A^{i}\in\mathsf{E}^{\kappa}\right\} .\label{fts}
\end{equation}
By $\bigsqcup_{i\in I}A^{i}$ we are denoting the function 
\[
f:i\in I\longmapsto A^{i}\in\mathsf{E}^{\kappa}.
\]
We shall say that $\bigsqcup_{i\in I}A^{i}$ is a \textit{term with
letters} $A^{i}$. The upper bond $2^{\left|\mathsf{E}^{\kappa}\right|}$
for the cardinality of letters in each term will be clear later (see
Remark \ref{recar}).\bigskip{}

Let us introduce some notation. Given a second term ${\textstyle \bigsqcup_{j\in J}}B^{j}$,
with related function $g:j\in J\longmapsto B^{j}\in\mathsf{E}^{\kappa}$,
we shall denote by
\begin{equation}
\left({\textstyle \bigsqcup_{i\in I}}A^{i}\right)\sqcup\left({\textstyle \bigsqcup_{j\in J}}B^{j}\right)\label{dn}
\end{equation}
the term defined by the function $h:I\vee J\rightarrow\mathsf{E}^{\mathbb{\kappa}}$
(on the disjoint union $I\vee J$) such that
\[
\left.h\right|_{I}=f\quad\textrm{and}\quad\left.h\right|_{J}=g.
\]
Note that, since $I\vee J=J\vee I$,
\begin{equation}
\left({\textstyle \bigsqcup_{i\in I}}A^{i}\right)\sqcup\left({\textstyle \bigsqcup_{j\in J}}B^{j}\right)=\left({\textstyle \bigsqcup_{j\in J}}B^{j}\right)\sqcup\left({\textstyle \bigsqcup_{i\in I}}A^{i}\right).\label{conm}
\end{equation}
In addition, given a third term ${\textstyle \bigsqcup_{k\in K}}C^{k}$,
since $\left(I\vee J\right)\vee K=I\vee\left(J\vee K\right)$, then
\begin{equation}
\left(\left({\textstyle \bigsqcup_{i\in I}}A^{i}\right)\sqcup\left({\textstyle \bigsqcup_{j\in J}}B^{j}\right)\right)\sqcup\left({\textstyle \bigsqcup_{k\in K}}C^{k}\right)=\left({\textstyle \bigsqcup_{i\in I}}A^{i}\right)\sqcup\left(\left({\textstyle \bigsqcup_{j\in J}}B^{j}\right)\sqcup\left({\textstyle \bigsqcup_{k\in K}}C^{k}\right)\right).\label{ass}
\end{equation}
In general, given a family of terms
\[
{\textstyle \bigsqcup_{i\in I_{j}}}A_{j}^{i},\quad j\in J,\quad\left|J\right|\leq2^{\left|\mathsf{E}^{\kappa}\right|},
\]
we shall denote ${\textstyle \bigsqcup_{j\in J}}\left({\textstyle \bigsqcup_{i\in I_{j}}}A_{j}^{i}\right)$
the term related to the function 
\[
h:\left(i,j\right)\in K\longmapsto A_{j}^{i}\in\mathsf{E}^{\kappa},
\]
with $K=\left\{ \left(i,j\right):j\in J\textrm{ and }i\in I_{j}\right\} $.
That is to say,
\begin{equation}
{\textstyle \bigsqcup_{j\in J}}\left({\textstyle \bigsqcup_{i\in I_{j}}}A_{j}^{i}\right)\coloneqq{\textstyle \bigsqcup_{\left(i,j\right)\in K}}A_{j}^{i}.\label{ji}
\end{equation}
Note that, since $\left|I_{j}\right|\leq2^{\left|\mathsf{E}^{\mathbb{\kappa}}\right|}$
for all $j$, then
\begin{equation}
\left|K\right|\leq\left|J\right|\cdot2^{\left|\mathsf{E}^{\mathbb{\kappa}}\right|}=\max\left\{ \left|J\right|,2^{\left|\mathsf{E}^{\mathbb{\kappa}}\right|}\right\} =2^{\left|\mathsf{E}^{\mathbb{\kappa}}\right|},\label{jiji}
\end{equation}
hence, the cardinality of $K$ is compatible with the definition of
$\mathsf{T}\left(\mathsf{E}^{\kappa}\right)$. If $I_{j}=I$ for all
$j$, then $K=I\times J$. \bigskip{}

For $I=\emptyset$, we have the \textit{empty} term, that we shall
also indicate $\emptyset$. Following above notation, it is clear
that
\begin{equation}
\left({\textstyle \bigsqcup_{i\in I}}A^{i}\right)\sqcup\emptyset=\emptyset\sqcup\left({\textstyle \bigsqcup_{i\in I}}A^{i}\right)={\textstyle \bigsqcup_{i\in I}}A^{i}.\label{avva}
\end{equation}

\subsection{\label{tap} Terms and propositions}

Given 
\[
A=\left(a_{1},\ldots,a_{n},\ldots\right)\in\mathsf{E}^{\kappa}\quad\textrm{and}\quad B=\left(b_{1},\ldots,b_{n},\ldots\right)\in\mathsf{E}^{\kappa},
\]
which represent two particular propositions of the $\kappa$-times
repeated experiment, we want $A\sqcup B$ to represent (or unless
to be related to) the proposition ``$A$ or $B$'' of the same experiment.
In particular, if $A=B$, the element $A\sqcup A$ would represent
the proposition ``$A$ or $A$'', which we should identify with
$A$; i.e. we should identify the elements $A\sqcup A$ and $A$.
Moreover, if for instance $A\subseteq B$ (component-wise order),
we should identify $A\sqcup B$ with $B$. Let us try to justify this.
If $a_{i}\subseteq b_{i}$ for all $i\in\kappa$, the fact that $a_{i}$
is true implies that $b_{i}$ is true too, for all $i$. So, if the
proposition ``$A$ or $B$'' is true, then $a_{i}$ is true for
all $i$ or $b_{i}$ is true for all $i$, and, according to the last
sentence, $b_{i}$ is true for all $i$. Consequently $B$ is true.
On the other hand, we would like that every element ${\textstyle \bigsqcup_{i\in I}}A^{i}$
such that $A^{i}=\hat{\mathbf{0}}$, for all $i\in I$, and also the
empty term, represent the absurd proposition. All this drives us to
consider the following equivalence relation on the set $\mathsf{T}\left(\mathsf{E}^{\kappa}\right)$:
\[
{\textstyle \bigsqcup_{i\in I}}A^{i}\backsim{\textstyle \bigsqcup_{j\in J}}B^{j}
\]
if and only if:
\begin{itemize}
\item for all $i\in I$,
\begin{itemize}
\item $\left(J\neq\emptyset\right)$ there exists $j\in J$ such that $A^{i}\subseteq B^{j}$;
\item $\left(J=\emptyset\right)$ $A^{i}=\hat{\mathbf{0}}$;
\end{itemize}
\item for all $j\in J$,
\begin{itemize}
\item $\left(I\neq\emptyset\right)$ there exists $i\in I$ such that $B^{j}\subseteq A^{i}$;
\item $\left(I=\emptyset\right)$ $B^{j}=\hat{\mathbf{0}}$.
\end{itemize}
\end{itemize}
It is clear that relation $\backsim$ is reflexive and symmetric.
The proof of transitivity is left to the reader. One of the classes
of $\backsim$ is formed out by the empty term $\emptyset$ and every
element of the form ${\textstyle \bigsqcup_{i\in I}}\hat{\mathbf{0}}$
(with $I\neq\emptyset$). In particular, we have that
\[
\emptyset\backsim\hat{\mathbf{0}}.
\]
The other classes, of course, are given by non-empty terms. 
\begin{rem}
\label{net} Note that, for every class, we can always take a non-empty
term as its representative, and we shall do it from now on. 
\end{rem}

To study the classes different from $\left[\hat{\mathbf{0}}\right]$,
we can use the following characterization: two elements ${\textstyle \bigsqcup_{i\in I}}A^{i}$
and ${\textstyle \bigsqcup_{j\in J}}B^{j}$, with $I,J\neq\emptyset$,
are in the same class if and only if:
\begin{enumerate}
\item for all $i\in I$, there exists $j\in J$ such that $A^{i}\subseteq B^{j}$;
\item for all $j\in J$, there exists $i\in I$ such that $B^{j}\subseteq A^{i}$.
\end{enumerate}
As a direct consequence of above characterization, given $A,B\in\mathsf{E}^{\kappa}$,
we have that
\begin{equation}
A\backsim B\;\Longleftrightarrow\;A=B.\label{abeq}
\end{equation}
Also, it is easy to show that $A\sqcup A\backsim A$ for all $A\in\mathsf{E}^{\kappa}$,
and more generally, given $B\in\mathsf{E}^{\kappa}$ such that $A\subseteq B$,
we have that $A\sqcup B\backsim B$. In terms of classes, we can say
that 
\[
\left[A\sqcup A\right]=\left[A\right]
\]
and
\[
\left[A\sqcup B\right]=\left[B\right]\quad\textrm{if}\quad A\subseteq B.
\]
More generally, we have the next result.
\begin{prop}
\label{basaa} Given ${\textstyle \bigsqcup_{i\in I}}A^{i}\in\mathsf{T}\left(\mathsf{E}^{\kappa}\right)$,
with $I\neq\emptyset$,
\begin{equation}
\left[{\textstyle \bigsqcup_{i\in I'}}A^{i}\sqcup{\textstyle \bigsqcup_{i\in I}}A^{i}\right]=\left[{\textstyle \bigsqcup_{i\in I}}A^{i}\right],\quad\forall I'\subseteq I,\label{aaa}
\end{equation}
and given in addition ${\textstyle \bigsqcup_{j\in J}}B^{j}\in\mathsf{T}\left(\mathsf{E}^{\kappa}\right)$,
also with $J\neq\emptyset$, we have that
\begin{equation}
\left[{\textstyle \bigsqcup_{i\in I}}A^{i}\sqcup{\textstyle \bigsqcup_{j\in J}}B^{j}\right]=\left[{\textstyle \bigsqcup_{j\in J}}B^{j}\right]\label{abb}
\end{equation}
if and only if for all $i\in I$ there exists $j\in J$ such that
$A^{i}\subseteq B^{j}$.
\end{prop}

\begin{proof}
If $I'=\emptyset$, since 
\[
{\textstyle \bigsqcup_{i\in I'}}A^{i}\sqcup{\textstyle \bigsqcup_{i\in I}}A^{i}=\emptyset\sqcup{\textstyle \bigsqcup_{i\in I}}A^{i}={\textstyle \bigsqcup_{i\in I}}A^{i}
\]
(see \eqref{avva}), Eq. \eqref{aaa} follows from reflexivity of
$\backsim$; and if $I'\neq\emptyset$, it follows from the fact that
$A^{i}\subseteq A^{i}$, for all $i\in I$, and the points $1$ and
$2$ above. To prove \eqref{abb}, we can use again the mentioned
points.
\end{proof}
\begin{rem}
\label{ecar} Because of Eq. \eqref{aaa}, it is clear that if ${\textstyle \bigsqcup_{i\in I}}A^{i}\in\mathsf{T}\left(\mathsf{E}^{\kappa}\right)$
is defined by the function $f$, then its class $\left[{\textstyle \bigsqcup_{i\in I}}A^{i}\right]$
only depends on $\mathsf{range}\left(f\right)$ (and not on the multiplicities
$\left|f^{-1}\left(l\right)\right|$, with $l\in\mathsf{range}\left(f\right)$).
In other words, each class $\left[{\textstyle \bigsqcup_{i\in I}}A^{i}\right]$
can be described in terms of a subset of $\mathsf{E}^{\kappa}$: the
different letters of the representing term. Thus,
\[
\left|\left.\mathsf{T}\left(\mathsf{E}^{\kappa}\right)\right/\backsim\right|\leq2^{\left|\mathsf{E}^{\kappa}\right|}.
\]
\end{rem}

\section{The completely distributive lattice $\mathsf{D}_{\kappa}\left(\mathsf{E}\right)$}

\label{tcdis} Let us define
\[
\mathsf{D}_{\kappa}\left(\mathsf{E}\right)\coloneqq\left.\mathsf{T}\left(\mathsf{E}^{\kappa}\right)\right/\backsim.
\]
In this section, we shall show that a natural completely distributive
lattice structure can be defined on $\mathsf{D}_{\kappa}\left(\mathsf{E}\right)$.
Let us begin with the join operation.

\subsection{\label{jop} The join operation and the related order}

Recall that we shall represent the classes of $\backsim$, i.e. the
elements of $\mathsf{D}_{\kappa}\left(\mathsf{E}\right)$, by non-empty
terms (see Remark \ref{net}). So, when we talk about ``the class
$\left[{\textstyle \bigsqcup_{i\in I}}A^{i}\right]$,'' it will be
implicit that $I\neq\emptyset$.
\begin{prop}
The formula 
\[
\left[{\textstyle \bigsqcup_{i\in I}}A^{i}\right]\vee\left[{\textstyle \bigsqcup_{j\in J}}B^{j}\right]\coloneqq\left[{\textstyle \bigsqcup_{i\in I}}A^{i}\sqcup{\textstyle \bigsqcup_{j\in J}}B^{j}\right]
\]
defines a join operation $\vee$ on $\mathsf{D}_{\kappa}\left(\mathsf{E}\right)$.
\end{prop}

\begin{proof}
Let us first see that $\vee$ is well-defined. Suppose that 
\[
\left[{\textstyle \bigsqcup_{i\in I}}A^{i}\right]=\left[{\textstyle \bigsqcup_{k\in K}}C^{k}\right].
\]
According to the points $1$ and $2$ above, for all $i\in I$ there
exists $k\in K$ such that $A^{i}\subseteq C^{k}$, and vice versa.
As a consequence, the terms
\[
{\textstyle \bigsqcup_{i\in I}}A^{i}\sqcup{\textstyle \bigsqcup_{j\in J}}B^{j}\quad\textrm{and}\quad{\textstyle \bigsqcup_{k\in K}}C^{k}\sqcup{\textstyle \bigsqcup_{j\in J}}B^{j}
\]
also satisfy the points $1$ and $2$, and consequently
\[
\left[{\textstyle \bigsqcup_{i\in I}}A^{i}\sqcup{\textstyle \bigsqcup_{j\in J}}B^{j}\right]=\left[{\textstyle \bigsqcup_{k\in K}}C^{k}\sqcup{\textstyle \bigsqcup_{j\in J}}B^{j}\right].
\]
All this ensures that $\vee$ is well-defined. Commutativity and associativity
of $\vee$ is a consequence of \eqref{conm} and \eqref{ass}, respectively,
while idempotency is a consequence of \eqref{aaa} for $I=I'$. 
\end{proof}
From $\vee$ we can define an order $\leq$ on $\mathsf{D}_{\kappa}\left(\mathsf{E}\right)$:
\[
\left[{\textstyle \bigsqcup_{i\in I}}A^{i}\right]\leq\left[{\textstyle \bigsqcup_{j\in J}}B^{j}\right]\quad\Longleftrightarrow\quad\left[{\textstyle \bigsqcup_{i\in I}}A^{i}\right]\vee\left[{\textstyle \bigsqcup_{j\in J}}B^{j}\right]=\left[{\textstyle \bigsqcup_{j\in J}}B^{j}\right],
\]
or equivalently
\begin{equation}
\left[{\textstyle \bigsqcup_{i\in I}}A^{i}\right]\leq\left[{\textstyle \bigsqcup_{j\in J}}B^{j}\right]\quad\Longleftrightarrow\quad\left[{\textstyle \bigsqcup_{i\in I}}A^{i}\sqcup{\textstyle \bigsqcup_{j\in J}}B^{j}\right]=\left[{\textstyle \bigsqcup_{j\in J}}B^{j}\right].\label{ord}
\end{equation}
With respect to this order, we have that 
\[
\left[\hat{\mathbf{0}}\right]\leq\left[{\textstyle \bigsqcup_{i\in I}}A^{i}\right]\leq\left[\hat{\mathbf{1}}\right],\quad\forall\left[{\textstyle \bigsqcup_{i\in I}}A^{i}\right]\in\mathsf{D}_{\kappa}\left(\mathsf{E}\right).
\]
In fact, it is easy to show that
\[
\left[\hat{\mathbf{0}}\sqcup{\textstyle \bigsqcup_{i\in I}}A^{i}\right]=\left[{\textstyle \bigsqcup_{i\in I}}A^{i}\right]\quad\textrm{and}\quad\left[{\textstyle \bigsqcup_{i\in I}}A^{i}\sqcup\hat{\mathbf{1}}\right]=\left[\hat{\mathbf{1}}\right].
\]
So, we have a bounded join semi-lattice $\left(\mathsf{D}_{\kappa}\left(\mathsf{E}\right),\leq,\vee,\left[\hat{\mathbf{0}}\right],\left[\hat{\mathbf{1}}\right]\right)$.

\bigskip{}

As expected, the order $\leq$ of $\mathsf{D}_{\kappa}\left(\mathsf{E}\right)$
and the (component-wise) order $\subseteq$ of $\mathsf{E}^{\kappa}$
are intimately related, as shown in the next result.
\begin{prop}
The function 
\begin{equation}
\varphi_{\kappa}:A\in\mathsf{E}^{\kappa}\mapsto\left[A\right]\in\mathsf{D}_{\kappa}\left(\mathsf{E}\right)\label{emb}
\end{equation}
defines an order embedding $\left(\mathsf{E}^{\kappa},\subseteq\right)\hookrightarrow\left(\mathsf{D}_{\kappa}\left(\mathsf{E}\right),\leq\right)$. 
\end{prop}

\begin{proof}
Given $A,B\in\mathsf{E}^{\kappa}$, we know from Proposition \ref{basaa}
(see Eq. \eqref{abb}) that $\left[A\sqcup B\right]=\left[B\right]$
if and only if $A\subseteq B$. But $\left[A\sqcup B\right]=\left[B\right]$
means exactly that $\left[A\right]\leq\left[B\right]$ (see Eq. \eqref{ord}),
what ends our proof.
\end{proof}
The injectivity of $\varphi_{\kappa}$ ensures that, given $A,B\in\mathsf{E}^{\kappa}$,
\[
\left[A\right]=\left[B\right]\;\Longleftrightarrow\;A=B,
\]
which is precisely what Eq. \eqref{abeq} says.

\subsection{\label{compss} Completeness}

Consider a non-empty subset
\[
S=\left\{ \left[{\textstyle \bigsqcup_{i\in I_{j}}}A_{j}^{i}\right]:j\in J\right\} \subseteq\mathsf{D}_{\kappa}\left(\mathsf{E}\right).
\]
Note that $J\neq\emptyset$. (According to our conventions, we shall
assume that $I_{j}\neq\emptyset$ for all $j$). Let us show that
$\bigvee S$ exists and it is given by (recall \eqref{ji})
\[
\bigvee S=\left[{\textstyle \bigsqcup_{j\in J}}\left({\textstyle \bigsqcup_{i\in I_{j}}}A_{j}^{i}\right)\right].
\]

\begin{rem}
\label{recar} Note that $\left|I_{j}\right|\leq2^{\left|\mathsf{E^{\kappa}}\right|}$,
for all $j\in J$ (by the very definition of $\mathsf{T}\left(\mathsf{E^{\kappa}}\right)$),
and (see Remark \ref{ecar})
\[
\left|J\right|=\left|S\right|\leq\left|\mathsf{D}_{\kappa}\left(\mathsf{E}\right)\right|\leq2^{\left|\mathsf{E^{\kappa}}\right|}.
\]
Then, according to \eqref{jiji}, ${\textstyle \bigsqcup_{j\in J}}\left({\textstyle \bigsqcup_{i\in I_{j}}}A_{j}^{i}\right)$
is effectively an element of $\mathsf{T}\left(\mathsf{E^{\kappa}}\right)$.
This is why, in the definition of $\mathsf{T}\left(\mathsf{E}^{\kappa}\right)$,
we did not consider terms with more than $2^{\left|\mathsf{E}^{\kappa}\right|}$
letters.
\end{rem}

For simplicity, let us write $\mathbb{A}_{j}\coloneqq{\textstyle \bigsqcup_{i\in I_{j}}}A_{j}^{i}$.
It is clear that (according to \eqref{aaa})
\[
\left[\mathbb{A}_{l}\sqcup\left({\textstyle \bigsqcup_{j\in J}}\mathbb{A}_{j}\right)\right]=\left[{\textstyle \bigsqcup_{j\in J}}\mathbb{A}_{j}\right],\quad\forall l\in J,
\]
and consequently
\[
\left[\mathbb{A}_{l}\right]\leq\left[{\textstyle \bigsqcup_{j\in J}}\mathbb{A}_{j}\right],\quad\forall l\in J,
\]
i.e. $\left[{\textstyle \bigsqcup_{j\in J}}\mathbb{A}_{j}\right]$
is an upper bond of $S$. If for some 
\[
{\textstyle \mathbb{C}=\bigsqcup_{k\in K}C^{k}\in\mathsf{T}\left(\mathsf{E^{\kappa}}\right)},
\]
with $K\neq\emptyset$, we have that
\[
\left[\mathbb{A}_{j}\right]\leq\left[\mathbb{C}\right],\quad\forall j\in J,
\]
or equivalently
\[
\left[\mathbb{A}_{j}\sqcup\mathbb{C}\right]=\left[\mathbb{C}\right],\quad\forall j\in J,
\]
then, according to Proposition \ref{basaa} (see Eq. \eqref{abb}),
for all $j\in J$ and all $i\in I_{j}$ there exists $k\in K$ such
that $A_{j}^{i}\subseteq C^{k}$, what implies that
\[
\left[\left({\textstyle \bigsqcup_{j\in J}}\mathbb{A}_{j}\right)\sqcup\mathbb{C}\right]=\left[\mathbb{C}\right],
\]
i.e.
\[
\left[{\textstyle \bigsqcup_{j\in J}}\mathbb{A}_{j}\right]\leq\left[\mathbb{C}\right].
\]
This says precisely that $\left[{\textstyle \bigsqcup_{j\in J}}\mathbb{A}_{j}\right]$
is the least upper bound of $S$, as we wanted to show. Then, the
join semi-lattice $\left(\mathsf{D}_{\kappa}\left(\mathsf{E}\right),\vee\right)$
is complete and bounded below by $\left[\hat{\mathbf{0}}\right]$.
As it is well-known, this implies that, given any non-empty subset
$S\subseteq\mathsf{D}_{\kappa}\left(\mathsf{E}\right)$, the greatest
lower bound $\bigwedge S$ exists and is given by
\[
\bigwedge S=\bigvee S^{l},
\]
where $S^{l}\subseteq\mathsf{D}_{\kappa}\left(\mathsf{E}\right)$
is the (non-empty) subset
\begin{equation}
S^{l}\coloneqq\left\{ x\in\mathsf{D}_{\kappa}\left(\mathsf{E}\right)\;:\;x\leq s,\;\forall s\in S\right\} .\label{sl-1}
\end{equation}
In particular, this defines a meet operation $\wedge$ on $\mathsf{D}_{\kappa}\left(\mathsf{E}\right)$. 

\bigskip{}

As usual, we shall take
\begin{equation}
\bigvee\emptyset=\left[\hat{\mathbf{0}}\right]\quad\textrm{and}\quad\bigwedge\emptyset=\left[\hat{\mathbf{1}}\right].\label{sup0}
\end{equation}

Concluding, we have proved the next theorem.
\begin{thm}
$\mathsf{D}_{\kappa}\left(\mathsf{E}\right)$ with the order $\leq$
(see Eq. \eqref{ord}) gives rise to the complete lattice
\[
\left(\mathsf{D}_{\kappa}\left(\mathsf{E}\right),\leq,\vee,\wedge,\left[\hat{\mathbf{0}}\right],\left[\hat{\mathbf{1}}\right]\right).
\]
\end{thm}

The order embedding $\varphi_{\kappa}:\mathsf{E}^{\kappa}\hookrightarrow\mathsf{D}_{\kappa}\left(\mathsf{E}\right)$
(see Eq. \eqref{emb}) enable us to see $\mathsf{E}^{\kappa}$ as
a subset of $\mathsf{D}_{\kappa}\left(\mathsf{E}\right)$, and consequently
to write $\left[A\right]=A$ for all $A\in\mathsf{E}^{\kappa}$. In
particular, we can write
\[
\left[\hat{\mathbf{0}}\right]=\hat{\mathbf{0}}\quad\textrm{and}\quad\left[\hat{\mathbf{1}}\right]=\hat{\mathbf{1}}.
\]
On the other hand, for every element $\left[{\textstyle \bigsqcup_{i\in I}}A^{i}\right]\in\mathsf{D}_{\kappa}\left(\mathsf{E}\right)$,
we have that
\[
\left[{\textstyle \bigsqcup_{i\in I}}A^{i}\right]={\textstyle \bigvee_{i\in I}}\left[A^{i}\right].
\]
So, we can write the elements of $\mathsf{D}_{\kappa}\left(\mathsf{E}\right)$
as 
\[
{\textstyle \bigvee_{i\in I}}A^{i},\quad\textrm{with}\quad I\neq\emptyset\quad\textrm{and}\quad A_{i}\in\mathsf{E}^{\kappa},\;\forall i\in I,
\]
and we will do it from now on.

\subsection{\label{inff} An infimum formula for $\mathsf{D}_{\kappa}\left(\mathsf{E}\right)$}

Given a non-empty subset $S\subseteq\mathsf{D}_{\kappa}\left(\mathsf{E}\right)$,
let us calculate $\bigwedge S$. We begin with the particular case
in which $S\subseteq\mathsf{E}^{\kappa}$ (using above mentioned identifications),
i.e.
\[
S=\left\{ A_{j}\in\mathsf{E}^{\kappa}:j\in J\right\} .
\]
Since the component-wise meet ${\textstyle \bigcap_{j\in J}A_{j}}\in\mathsf{E}^{\kappa}$
satisfies
\[
{\textstyle \bigcap_{j\in J}A_{j}}\subseteq A_{l},\quad\forall l\in I,
\]
then, using the embedding $\varphi_{\kappa}$ (see Eq. \eqref{emb}),
\[
{\textstyle \bigcap_{j\in J}A_{j}}\leq A_{l},\quad\forall l\in I,
\]
i.e. $\bigcap_{j\in J}A_{j}$ is a lower bound of $S$. On the other
hand, an element ${\textstyle \bigvee_{k\in K}}C^{k}$ is a lower
bound of $S$, i.e. 
\[
{\textstyle \bigvee_{k\in K}}C^{k}\leq A_{j},\quad\forall j\in J,
\]
or equivalently
\[
\left[{\textstyle \bigsqcup_{k\in K}}C^{k}\sqcup A_{j}\right]=\left[A_{j}\right],\quad\forall j\in J,
\]
if and only if (see Eq. \eqref{abb}) for all $k\in K$ we have that
$C^{k}\subseteq A_{j}$, for all $j\in J$. Thus, for all $k\in K$,
\[
{\textstyle C^{k}\subseteq\bigcap_{j\in J}A_{j}},
\]
so
\[
{\textstyle \left[{\textstyle \bigsqcup_{k\in K}}C^{k}\sqcup\bigcap_{j\in J}A_{j}\right]=\left[\bigcap_{j\in J}A_{j}\right]}
\]
and consequently
\[
{\textstyle {\textstyle \bigvee_{k\in K}}C^{k}\leq\bigcap_{j\in J}A_{j}.}
\]
All this implies that $\bigcap_{j\in J}A_{j}$ is the greatest lower
bound of $S$, i.e.
\begin{equation}
{\textstyle \bigwedge_{j\in J}A_{j}=\bigcap_{j\in J}A_{j}}.\label{vu}
\end{equation}

\bigskip{}

Now, let us study the general case. Consider a non-empty subset

\[
S=\left\{ {\textstyle \bigvee_{i\in I_{j}}}A_{j}^{i}:j\in J\right\} \subseteq\mathsf{D}_{\kappa}\left(\mathsf{E}\right),
\]
and let us calculate 
\[
{\textstyle \bigwedge S=\bigwedge_{j\in J}{\textstyle \bigvee_{i\in I_{j}}}A_{j}^{i}.}
\]
As happens in any complete lattice, we have that 
\begin{equation}
{\textstyle \bigvee_{f\in F}\bigwedge_{j\in J}A_{j}^{f\left(j\right)}\leq\bigwedge_{j\in J}\bigvee_{i\in I_{j}}A_{j}^{i}}\label{minmax}
\end{equation}
(i.e. the \textsl{Min-Max Property} holds), where 
\begin{equation}
F=\left\{ \left.{\textstyle f:J\rightarrow\bigcup_{j\in J}I_{j}}\quad\right/\quad f\left(j\right)\in I_{j},\;\forall j\in J\right\} .\label{F}
\end{equation}
In particular,
\[
{\textstyle \bigvee_{f\in F}\bigwedge_{j\in J}A_{j}^{f\left(j\right)}}
\]
is a lower bound of $S$. Let us see that it is the greatest one.
Firstly note that, according to \eqref{vu}, 
\[
{\textstyle \bigvee_{f\in F}\bigwedge_{j\in J}A_{j}^{f\left(j\right)}}={\textstyle \bigvee_{f\in F}\bigcap_{j\in J}A_{j}^{f\left(j\right)}}.
\]
Secondly, suppose that ${\textstyle \bigvee_{k\in K}}C^{k}$ is another
lower bound of $S$, i.e. 
\[
{\textstyle \bigvee_{k\in K}}C^{k}\leq{\textstyle \bigvee_{i\in I_{j}}}A_{j}^{i},\quad\forall j\in J,
\]
or equivalently
\[
\left[{\textstyle \bigsqcup_{k\in K}}C^{k}\sqcup{\textstyle \bigsqcup_{i\in I_{j}}}A_{j}^{i}\right]=\left[{\textstyle \bigsqcup_{i\in I_{j}}}A_{j}^{i}\right],\quad\forall j\in J,
\]
what is possible if and only if (see Eq. \eqref{abb}) for all $j\in J$
and all $k\in K$ there exists $i\in I_{j}$ such that $C^{k}\subseteq A_{j}^{i}$.
Thus, for all $k\in K$ there exists a function 
\[
f_{k}:J\rightarrow\cup_{j\in J}I_{j}
\]
such that $f_{k}\left(j\right)\in I_{j}$ and 
\[
C^{k}\subseteq A_{j}^{f_{k}\left(j\right)},\quad\forall j\in J,
\]
or equivalently ,
\[
{\textstyle C^{k}\subseteq\bigcap_{j\in J}A_{j}^{f_{k}\left(j\right)}.}
\]
Using the embedding $\varphi_{\kappa}$, last equation says that 
\[
{\textstyle C^{k}\leq\bigcap_{j\in J}A_{j}^{f_{k}\left(j\right)},}\quad\forall k\in K,
\]
and consequently (since the functions $f_{k}$'s define a subset of
$F$)
\[
{\textstyle {\textstyle \bigvee_{k\in K}}C^{k}\leq{\textstyle \bigvee_{k\in K}}\bigcap_{j\in J}A_{j}^{f_{k}\left(j\right)}\leq{\textstyle \bigvee_{f\in F}}\bigcap_{j\in J}A_{j}^{f\left(j\right)},}
\]
This finally proves that 
\begin{equation}
{\textstyle \bigwedge S=\bigwedge_{j\in J}{\textstyle \bigvee_{i\in I_{j}}}A_{j}^{i}={\textstyle \bigvee_{f\in F}}\bigcap_{j\in J}A_{j}^{f\left(j\right)}.}\label{finf}
\end{equation}

\subsection{\label{cdis} Complete distributivity}

In this section, we shall show that 
\begin{equation}
{\textstyle \bigwedge_{j\in J}{\textstyle \bigvee_{i\in I_{j}}}\mathbb{A}_{j}^{i}={\textstyle \bigvee_{f\in F}}\bigwedge_{j\in J}\mathbb{A}_{j}^{f\left(j\right)}},\label{comdis}
\end{equation}
for every family of elements $\mathbb{A}_{j}^{i}\in\mathsf{D}_{\kappa}\left(\mathsf{E}\right)$.
According to Ref. \cite{raney}, this is true if and only if so is
the dual identity:
\begin{equation}
{\textstyle \bigvee_{j\in J}{\textstyle \bigwedge_{i\in I_{j}}}\mathbb{A}_{j}^{i}={\textstyle \bigwedge_{f\in F}}\bigvee_{j\in J}\mathbb{A}_{j}^{f\left(j\right)}.}\label{comdisd}
\end{equation}
A complete lattice satisfying Eqs. \eqref{comdis} and \eqref{comdisd}
is said to be completely distributive.
\begin{thm}
The lattice $\left(\mathsf{D}_{\kappa}\left(\mathsf{E}\right),\leq,\vee,\wedge,\hat{\mathbf{0}},\hat{\mathbf{1}}\right)$
is completely distributive.
\end{thm}

\begin{proof}
It is enough to prove the validity of Eq. \eqref{comdis}, which,
writing $\mathbb{A}_{j}^{i}=\bigvee_{k\in K_{ij}}A_{j}^{\left(i,k\right)}$,
for some set $K_{ij}$ and some $A_{j}^{i,k}\in\mathsf{E}^{\kappa}$,
it translates to
\begin{equation}
{\textstyle \bigwedge_{j\in J}{\textstyle \bigvee_{i\in I_{j}}}\bigvee_{k\in K_{ij}}A_{j}^{\left(i,k\right)}={\textstyle \bigvee_{f\in F}}\bigwedge_{j\in J}\bigvee_{k\in K_{f\left(j\right)j}}A_{j}^{\left(f\left(j\right),k\right)}}.\label{e}
\end{equation}
To prove above equation, let us write
\begin{equation}
{\textstyle {\textstyle \bigvee_{i\in I_{j}}}\bigvee_{k\in K_{ij}}A_{j}^{\left(i,k\right)}={\textstyle \bigvee_{l\in L_{j}}}A_{j}^{\left[l\right]}},\label{a}
\end{equation}
where
\[
{\textstyle L_{j}=\left\{ \left(i,k\right)\;:\;i\in I_{j},\;k\in K_{ij}\right\} ,\quad j\in J.}
\]
We know from the results of the last section that
\begin{equation}
{\textstyle \bigwedge_{j\in J}{\textstyle \bigvee_{l\in L_{j}}}A_{j}^{\left[l\right]}={\textstyle \bigvee_{h\in H}}\bigcap_{j\in J}A_{j}^{\left[h\left(j\right)\right]},}\label{b}
\end{equation}
with 
\[
H=\left\{ h:J\rightarrow\cup_{j\in J}L_{j}\;/\;h\left(j\right)\in L_{j}\right\} .
\]
Note that, according to the definitions of $L_{j}$ and $H$, we can
write
\[
h\left(j\right)=\left(f\left(j\right),g\left(j\right)\right),\quad j\in J,
\]
where $f$ is a function inside the set $F$ (see Eq. \eqref{F}),
and $g$ is inside 
\[
G_{f}=\left\{ g:J\rightarrow\cup_{j\in J}K_{f\left(j\right)j}\;:\;g\left(j\right)\in K_{f\left(j\right)j}\right\} .
\]
Then
\begin{equation}
{\textstyle {\textstyle \bigvee_{h\in H}}\bigcap_{j\in J}A_{j}^{\left[h\left(j\right)\right]}={\textstyle \bigvee_{f\in F}}{\textstyle \bigvee_{g\in G_{f}}}\bigcap_{j\in J}A_{j}^{\left(f\left(j\right),g\left(j\right)\right)}},\label{c}
\end{equation}
and, using again the results of the last section, for each $f\in F$,
\begin{equation}
{\textstyle {\textstyle \bigvee_{g\in G_{f}}}\bigcap_{j\in J}A_{j}^{\left(f\left(j\right),g\left(j\right)\right)}=\bigwedge_{j\in J}\bigvee_{k\in K_{f\left(j\right)j}}A_{j}^{\left(f\left(j\right),k\right)}}.\label{d}
\end{equation}
Finally, combining the Eqs. \eqref{a} to \eqref{d}, we have precisely
Eq. \eqref{e}, as wanted.
\end{proof}

\section{The universal logic of repeated experiments}

\label{unlo} 

\subsection{The negation operation}

We want to define a \textit{negation} for each proposition of $\mathsf{D}_{\kappa}\left(\mathsf{E}\right)$\textsl{.}
(Recall that $^{c}$ denotes the orthocomplementation of $\mathsf{E}$).
Given 
\[
A=\left(a_{1},\ldots,a_{n},\ldots\right)\in\mathsf{E}^{\kappa}\subseteq\mathsf{D}_{\kappa}\left(\mathsf{E}\right),
\]
a natural candidate for its negation could be 
\[
{\textstyle A^{*}=\left[\bigsqcup_{n\in\kappa}\left(\mathbf{1},\ldots,\mathbf{1},a_{n}^{c},\mathbf{1},\ldots\right)\right]=\bigvee_{n\in\kappa}\left(\mathbf{1},\ldots,\mathbf{1},a_{n}^{c},\mathbf{1},\ldots\right)},
\]
because $\left(\mathbf{1},\ldots,\mathbf{1},a_{n}^{c},\mathbf{1},\ldots\right)$
represents the proposition: ``$a_{n}$ does not occur in the $n$-th
repetition.'' In particular, 
\begin{equation}
\hat{\mathbf{0}}^{*}=\bigvee_{n\in\kappa}(\underbrace{\mathbf{1},\ldots,\mathbf{1},\mathbf{0}^{c}}_{n},\mathbf{1},\ldots)=\bigvee_{n\in\kappa}\left(\mathbf{1},\ldots,\mathbf{1},\ldots\right)=\left(\mathbf{1},\ldots,\mathbf{1},\ldots\right)=\hat{\mathbf{1}},\label{01}
\end{equation}
and
\begin{equation}
\hat{\mathbf{1}}^{*}=\bigvee_{n\in\kappa}(\underbrace{\mathbf{1},\ldots,\mathbf{1},\mathbf{1}^{c}}_{n},\mathbf{1},\ldots)=\bigvee_{n\in\kappa}\left(\mathbf{0},\ldots,\mathbf{0},\ldots\right)=\left(\mathbf{0},\ldots,\mathbf{0},\ldots\right)=\hat{\mathbf{0}}.\label{1s0}
\end{equation}
And for an arbitrary element ${\textstyle \bigvee_{i\in I}}A^{i}\in\mathsf{D}_{\kappa}\left(\mathsf{E}\right)$
(with $I\neq\emptyset$), the natural candidate would be

\begin{equation}
{\textstyle \left({\textstyle \bigvee_{i\in I}}A^{i}\right)^{*}={\textstyle \bigwedge_{i\in I}}\left(A^{i}\right)^{*}=\bigwedge_{i\in I}\bigvee_{n\in\kappa}\left(\mathbf{1},\ldots,\mathbf{1},\left(a_{n}^{i}\right)^{c},\mathbf{1},\ldots\right)}.\label{dor}
\end{equation}
Using the complete distributivity, note that
\[
{\textstyle \bigwedge_{i\in I}\bigvee_{n\in\kappa}\left(\mathbf{1},\ldots,\mathbf{1},\left(a_{n}^{i}\right)^{c},\mathbf{1},\ldots\right)={\textstyle \bigvee_{f\in F}}\bigcap_{i\in I}\left(\mathbf{1},\ldots,\mathbf{1},\left(a_{f\left(i\right)}^{i}\right)^{c},\mathbf{1},\ldots\right)},
\]
with $F=\left\{ f:I\rightarrow\kappa\right\} $, and since 
\[
{\textstyle \bigcap_{i\in I}=\bigcap_{n\in\kappa}\bigcap_{i\in f^{-1}\left(n\right)}},
\]
we have that
\begin{equation}
\begin{array}{lll}
\bigwedge_{i\in I}\bigvee_{n\in\kappa}\left(\mathbf{1},\ldots,\mathbf{1},\left(a_{n}^{i}\right)^{c},\mathbf{1},\ldots\right) & = & {\textstyle \bigvee_{f\in F}\bigcap_{n\in\kappa}\bigcap_{i\in f^{-1}\left(n\right)}}\left(\mathbf{1},\ldots,\mathbf{1},\left(a_{n}^{i}\right)^{c},\mathbf{1},\ldots\right)\\
\\
 & = & {\textstyle \bigvee_{f\in F}\bigcap_{n\in\kappa}}\left(\mathbf{1},\ldots,\mathbf{1},\bigcap_{i\in f^{-1}\left(n\right)}\left(a_{n}^{i}\right)^{c},\mathbf{1},\ldots\right)\\
\\
 & = & {\textstyle \bigvee_{f\in F}}\left(\bigcap_{i\in f^{-1}\left(1\right)}\left(a_{1}^{i}\right)^{c},\ldots,\bigcap_{i\in f^{-1}\left(n\right)}\left(a_{n}^{i}\right)^{c},\ldots\right).
\end{array}\label{dor1}
\end{equation}
That is to say, 
\begin{equation}
{\textstyle \left({\textstyle \bigvee_{i\in I}}A^{i}\right)^{*}={\textstyle \bigvee_{f\in F}}\left(\bigcap_{i\in f^{-1}\left(1\right)}\left(a_{1}^{i}\right)^{c},\ldots,\bigcap_{i\in f^{-1}\left(n\right)}\left(a_{n}^{i}\right)^{c},\ldots\right)}\label{dor2}
\end{equation}
would be a good candidate. (Recall Eqs. \eqref{fe2} and \eqref{fen}).
Let us see that above formula gives rise to a well-defined function
from $\mathsf{D}_{\kappa}\left(\mathsf{E}\right)$ to $\mathsf{D}_{\kappa}\left(\mathsf{E}\right)$.
\begin{prop}
\label{orev} The function $*:\mathsf{D}_{\kappa}\left(\mathsf{E}\right)\rightarrow\mathsf{D}_{\kappa}\left(\mathsf{E}\right)$
given by
\begin{equation}
{\textstyle \left(\left[{\textstyle \bigsqcup_{i\in I}}A^{i}\right]\right)^{*}=\left[{\textstyle \bigsqcup_{f\in F}}\left(\bigcap_{i\in f^{-1}\left(1\right)}\left(a_{1}^{i}\right)^{c},\ldots,\bigcap_{i\in f^{-1}\left(n\right)}\left(a_{n}^{i}\right)^{c},\ldots\right)\right],}\label{dop}
\end{equation}
with $F=\left\{ f:I\rightarrow\kappa\right\} $, is well-defined and
it is order-reversing, i.e. for all $\mathbb{A},\mathbb{B}\in\mathsf{D}_{\kappa}\left(\mathsf{E}\right)$
such that $\mathbb{A}\leq\mathbb{B}$, we have that $\mathbb{B}^{*}\leq\mathbb{A}^{*}$. 
\end{prop}

\begin{proof}
Consider $\mathbb{A}=\left[{\textstyle \bigsqcup_{i\in I}}A^{i}\right]$
and $\mathbb{B}=\left[{\textstyle \bigsqcup_{j\in J}}B^{j}\right]$
(with $I,J\neq\emptyset$, by our convention) such that $\mathbb{A}\leq\mathbb{B}$,
i.e. 
\[
\left[{\textstyle \bigsqcup_{i\in I}}A^{i}\sqcup{\textstyle \bigsqcup_{j\in J}}B^{j}\right]=\left[{\textstyle \bigsqcup_{j\in J}}B^{j}\right].
\]
 Let us show that 
\begin{equation}
\left[\bigsqcup_{g\in G}\left(\bigcap_{j\in g^{-1}\left(1\right)}\left(b_{1}^{j}\right)^{c},\ldots,\bigcap_{j\in g^{-1}\left(n\right)}\left(b_{n}^{j}\right)^{c},\ldots\right)\right]\leq\left[\bigsqcup_{f\in F}\left(\bigcap_{i\in f^{-1}\left(1\right)}\left(a_{1}^{i}\right)^{c},\ldots,\bigcap_{i\in f^{-1}\left(n\right)}\left(a_{n}^{i}\right)^{c},\ldots\right)\right],\label{wwp}
\end{equation}
where the meaning of $G$ is analogous to that of $F$ (changing $I$
by $J$). Recall that (according to Proposition \ref{basaa}), for
each $i\in I$ there exists $j\in J$ such that $A^{i}\subseteq B^{j}$.
Let us choose $j=h\left(i\right)$, for some function $h:I\rightarrow J$.
Then, for each $i$, we have that $a_{n}^{i}\subseteq b_{n}^{h\left(i\right)}$,
for all $n\in\kappa$, or equivalently $\left(b_{n}^{h\left(i\right)}\right)^{c}\subseteq\left(a_{n}^{i}\right)^{c}$,
for all $n\in\kappa$, and, as a consequence, 
\[
\left(\mathbf{1},\ldots,\mathbf{1},\left(b_{n}^{h\left(i\right)}\right)^{c},\mathbf{1},\ldots\right)\subseteq\left(\mathbf{1},\ldots,\mathbf{1},\left(a_{n}^{i}\right)^{c},\mathbf{1},\ldots\right),\quad\forall i\in I,\quad\forall n\in\kappa.
\]
Then
\[
{\textstyle \bigvee_{n\in\kappa}\left(\mathbf{1},\ldots,\mathbf{1},\left(b_{n}^{h\left(i\right)}\right)^{c},\mathbf{1},\ldots\right)\leq\bigvee_{n\in\kappa}\left(\mathbf{1},\ldots,\mathbf{1},\left(a_{n}^{i}\right)^{c},\mathbf{1},\ldots\right),\quad\forall i\in I,}
\]
and
\[
{\textstyle \bigwedge_{i\in I}\bigvee_{n\in\kappa}\left(\mathbf{1},\ldots,\mathbf{1},\left(b_{n}^{h\left(i\right)}\right)^{c},\mathbf{1},\ldots\right)\leq\bigwedge_{i\in I}\bigvee_{n\in\kappa}\left(\mathbf{1},\ldots,\mathbf{1},\left(a_{n}^{i}\right)^{c},\mathbf{1},\ldots\right)}.
\]
Finally, since $\mathsf{Im}h\subseteq J$, we can conclude that

\[
{\textstyle \bigwedge_{j\in J}\bigvee_{n\in\kappa}\left(\mathbf{1},\ldots,\mathbf{1},\left(b_{n}^{j}\right)^{c},\mathbf{1},\ldots\right)\leq\bigwedge_{i\in I}\bigvee_{n\in\kappa}\left(\mathbf{1},\ldots,\mathbf{1},\left(a_{n}^{i}\right)^{c},\mathbf{1},\ldots\right)}.
\]
But, taking into account Eq. \eqref{dor1}, the members of last inequality
are exactly the members of \eqref{wwp}, so, the wanted result follows.
If in addition $\mathbb{B}\leq\mathbb{A}$, i.e. $\mathbb{A}=\mathbb{B}$,
or equivalently ${\textstyle \bigsqcup_{i\in I}}A^{i}$ and ${\textstyle \bigsqcup_{j\in J}}B^{j}$
are in the same class, the equality of such members holds and then,
according to \eqref{dop}, $*$ is well-defined. In particular, coming
back to the case in which $\mathbb{A}\leq\mathbb{B}$, above calculations
says that $\mathbb{B}^{*}\leq\mathbb{A}^{*}$, what shows that $*$
is order-reversing.
\end{proof}

\subsection{A closure operator on $\mathsf{D}_{\kappa}\left(\mathsf{E}\right)$}

Let us study some properties of $*$.
\begin{prop}
\label{mml} For all $A\in\mathsf{E}^{\kappa}$,
\begin{equation}
{\textstyle A^{*}=\bigvee_{n\in\kappa}\left(\mathbf{1},\ldots,\mathbf{1},a_{n}^{c},\mathbf{1},\ldots\right)},\label{dop2}
\end{equation}
and consequently
\begin{equation}
\hat{\mathbf{0}}^{*}=\hat{\mathbf{1}}\quad\textrm{and}\quad\hat{\mathbf{1}}^{*}=\hat{\mathbf{0}}.\label{dop3}
\end{equation}
And given any family $S=\left\{ \mathbb{A}_{j}:j\in J\right\} \subseteq\mathsf{D}_{\kappa}\left(\mathsf{E}\right)$,
we have that
\begin{equation}
{\textstyle \left(\bigvee_{j\in J}\mathbb{A}_{j}\right)^{*}=\bigwedge_{j\in J}\mathbb{A}_{j}^{*},}\label{dop4}
\end{equation}
or equivalently 
\[
\left(\bigvee S\right)^{*}=\bigwedge S^{*},\quad\textrm{where}\quad S^{*}=\left\{ \mathbb{A}_{j}^{*}:j\in J\right\} .
\]
\end{prop}

\begin{proof}
We must show \eqref{dop2} from the precise definition of $*$, which
is given by \eqref{dop} (or \eqref{dor2}). Note that, if $I$ is
a singleton $\left\{ s\right\} $, then we can define a bijection
\[
m\in\kappa\mapsto f_{m}\in F,
\]
where $f_{m}\left(s\right)=m$. Thus $f_{m}^{-1}\left(n\right)=\emptyset$
for all $n\neq m$, and \eqref{dop} translates exactly to \eqref{dop2}.
The Eq. \eqref{dop3} can be proved as in \eqref{01} and \eqref{1s0}.

Let us prove \eqref{dop4}. If $J=\emptyset$, the result follows
from \eqref{sup0} and \eqref{dop3}. Otherwise, let us first assume
that $\mathbb{A}_{j}=A_{j}\in\mathsf{E}^{\kappa}$ for all $j\in J$.
Then, combining \eqref{dor1}, \eqref{dor2} (or \eqref{dop}) and
\eqref{dop2}, we easily arrive at
\begin{equation}
{\textstyle \left(\bigvee_{j\in J}A_{j}\right)^{*}=\bigwedge_{j\in J}\left(A_{j}\right)^{*}.}\label{pc}
\end{equation}
Now consider the general case, and write ${\textstyle \mathbb{A}_{j}=\bigvee_{i\in I_{j}}A_{j}^{i}}$.
Then, 
\[
{\textstyle \left(\bigvee_{j\in J}\mathbb{A}_{j}\right)^{*}=\left(\bigvee_{j\in J}\bigvee_{i\in I_{j}}A_{j}^{i}\right)^{*}=\bigwedge_{j\in J}\bigwedge_{i\in I_{j}}\left(A_{j}^{i}\right)^{*}=\bigwedge_{j\in J}\left(\bigvee_{i\in I_{j}}A_{j}^{i}\right)^{*}=\bigwedge_{j\in J}\mathbb{A}_{j}^{*}},
\]
where we have used twice the result of the particular case \eqref{pc}.
\end{proof}
The function $*$ has some of the properties of an orthocomplementation,
but it is not one. For instance, given $A=\left(a_{1},\ldots,a_{n},\ldots\right)\in\mathsf{E}^{\kappa}$,
from \eqref{dop2} and \eqref{dop4} we have that
\begin{equation}
\begin{array}{lll}
A^{**} & = & \left(\bigvee_{n\in\kappa}\left(\mathbf{1},\ldots,\mathbf{1},a_{n}^{c},\mathbf{1},\ldots\right)\right)^{*}=\bigwedge_{n\in\kappa}\left(\mathbf{1},\ldots,\mathbf{1},a_{n}^{c},\mathbf{1},\ldots\right)^{*}\\
\\
 & = & \bigwedge_{n\in\kappa}\left(\mathbf{1},\ldots,\mathbf{1},a_{n},\mathbf{1},\ldots\right)=\left(a_{1},\ldots,a_{n},\ldots\right)=A.
\end{array}\label{assa}
\end{equation}
But this is not the case of a general element $\mathbb{A}\in\mathsf{D}_{\kappa}\left(\mathsf{E}\right)$.
In fact, for 
\[
A=\left(a,\hat{\mathbf{1}}\right)\quad\textrm{and}\quad B=\left(b,\hat{\mathbf{1}}\right),\quad\textrm{with}\quad a,b\in\mathsf{E},
\]
we have that
\begin{equation}
\begin{array}{lll}
\left(A\vee B\right)^{**} & = & \left(\left(a,\hat{\mathbf{1}}\right)\vee\left(b,\hat{\mathbf{1}}\right)\right)^{**}=\left(\left(a,\hat{\mathbf{1}}\right)^{*}\wedge\left(b,\hat{\mathbf{1}}\right)^{*}\right)^{*}=\left(\left(a^{c},\hat{\mathbf{1}}\right)\wedge\left(b^{c},\hat{\mathbf{1}}\right)\right)^{*}\\
\\
 & = & \left(a^{c}\cap b^{c},\hat{\mathbf{1}}\right)^{*}=\left(a\cup b,\hat{\mathbf{1}}\right)\neq A\vee B,
\end{array}\label{r4}
\end{equation}
in general. Note that, however, $A\vee B\leq\left(A\vee B\right)^{**}$.
\begin{prop}
For all $\mathbb{A}\in\mathsf{D}_{\kappa}\left(\mathsf{E}\right)$,
we have that $\mathbb{A}\leq\mathbb{A}^{**}$.
\end{prop}

\begin{proof}
Given $\mathbb{A}={\textstyle \bigvee_{i\in I}}A^{i}\in\mathsf{D}_{\kappa}\left(\mathsf{E}\right)$,
and using \eqref{dop2} and \eqref{dop4},
\[
\begin{array}{lll}
\left({\textstyle \bigvee_{i\in I}}A^{i}\right)^{**} & = & \left(\left({\textstyle \bigvee_{i\in I}}A^{i}\right)^{*}\right)^{*}=\left(\bigvee_{f\in F}\left(\bigcap_{i\in f^{-1}\left(1\right)}\left(a_{1}^{i}\right)^{c},\ldots,\bigcap_{i\in f^{-1}\left(n\right)}\left(a_{n}^{i}\right)^{c},\ldots\right)\right)^{*}\\
\\
 & = & \bigwedge_{f\in F}\left(\bigcap_{i\in f^{-1}\left(1\right)}\left(a_{1}^{i}\right)^{c},\ldots,\bigcap_{i\in f^{-1}\left(n\right)}\left(a_{n}^{i}\right)^{c},\ldots\right)^{*}\\
\\
 & = & \bigwedge_{f\in F}\bigvee_{n\in\kappa}\left(\mathbf{1},\ldots,\mathbf{1},\bigcup_{i\in f^{-1}\left(n\right)}a_{n}^{i},\mathbf{1},\ldots\right).
\end{array}
\]
Given $l\in I$ and $f\in F$, if $f\left(l\right)=m$, then $a_{m}^{l}\subseteq\bigcup_{i\in f^{-1}\left(m\right)}a_{m}^{i}$
and consequently
\[
A^{l}=\left(a_{1}^{l},\ldots,a_{m}^{l},\ldots\right)\leq\left(\mathbf{1},\ldots,\mathbf{1},\bigcup_{i\in f^{-1}\left(m\right)}a_{m}^{i},\mathbf{1},\ldots\right)\leq\bigvee_{n\in\kappa}\left(\mathbf{1},\ldots,\mathbf{1},\bigcup_{i\in f^{-1}\left(n\right)}a_{n}^{i},\mathbf{1},\ldots\right).
\]
Let us call $\mathbb{B}_{f}$ the last term of above equation. Then,
for each $l\in I$, the element $A^{l}$ is a lower bound of $\left\{ \mathbb{B}_{f}:f\in F\right\} $.
This implies that 
\[
{\textstyle A^{l}\leq\bigwedge_{f\in F}\bigvee_{n\in\kappa}\left(\mathbf{1},\ldots,\mathbf{1},\bigcup_{i\in f^{-1}\left(n\right)}a_{n}^{i},\mathbf{1},\ldots\right),\quad\forall l\in I,}
\]
and as a consequence 
\[
{\textstyle \bigvee_{l\in I}A^{l}\leq\bigwedge_{f\in F}\bigvee_{n\in\kappa}\left(\mathbf{1},\ldots,\mathbf{1},\bigcup_{i\in f^{-1}\left(n\right)}a_{n}^{i},\mathbf{1},\ldots\right)=\left({\textstyle \bigvee_{i\in I}}A^{i}\right)^{**}}
\]
This ends our proof.
\end{proof}
Using that 
\begin{equation}
\mathbb{A}\leq\mathbb{A}^{**},\label{op}
\end{equation}
for all $\mathbb{A}\in\mathsf{D}_{\kappa}\left(\mathsf{E}\right)$,
we have the inequality $\mathbb{A}^{*}\leq\left(\mathbb{A}^{*}\right)^{**}=\mathbb{A}^{***}$.
And using the order reversibility of $*$, from $\mathbb{A}\leq\mathbb{A}^{**}$
we have that $\mathbb{A}^{***}\leq\mathbb{A}^{*}$. Then 
\begin{equation}
\mathbb{A}^{*}=\mathbb{A}^{***}\label{31}
\end{equation}
and, applying $*$ in both members, it follows that 
\begin{equation}
\mathbb{A}^{****}=\mathbb{A}^{**}.\label{idp}
\end{equation}
All this proves the following result.
\begin{thm}
The function $**:\mathsf{D}_{\kappa}\left(\mathsf{E}\right)\rightarrow\mathsf{D}_{\kappa}\left(\mathsf{E}\right)$
is a \textbf{closure}, i.e. $**$ is order-preserving, satisfies \eqref{op}
and it is idempotent (see \eqref{idp}).
\end{thm}

\subsection{The logic $\mathsf{U}_{\kappa}\left(\mathsf{E}\right)$}

As it is well-known, given a closure $\rho:L\rightarrow L$ on a complete
lattice $L$, its image $\mathsf{Im}\rho\subseteq L$, which can be
described as
\[
\mathsf{Im}\rho=\left\{ x\in L:\rho\left(x\right)=x\right\} ,
\]
together with the inherited order, is also a complete lattice with
the same top and bottom elements (see Ref. \cite{blyth}). Infimum
coincide, i.e. given $S\subseteq\mathsf{Im}\rho$,
\[
{\textstyle \bigwedge_{\mathsf{Im}\rho}S=\bigwedge_{L}S},
\]
but for the supremum we have that
\[
{\textstyle \bigvee_{\mathsf{Im}\rho}S=\rho\left(\bigvee_{L}S\right).}
\]

For the closure operator $**:\mathsf{D}_{\kappa}\left(\mathsf{E}\right)\rightarrow\mathsf{D}_{\kappa}\left(\mathsf{E}\right)$,
besides a complete lattice structure on its image $\mathsf{D}_{\kappa}\left(\mathsf{E}\right)^{**}$,
we also have an orthocomplementation, and it is given by $*$. 
\begin{thm}
\label{pop} The subset 
\[
\mathsf{U}_{\kappa}\left(\mathsf{E}\right)\coloneqq\mathsf{D}_{\kappa}\left(\mathsf{E}\right)^{**}=\left\{ \mathbb{A}\in\mathsf{D}_{\kappa}\left(\mathsf{E}\right):\mathbb{A}^{**}=\mathbb{A}\right\} \subseteq\mathsf{D}_{\kappa}\left(\mathsf{E}\right)
\]
contains $\mathsf{E}^{\kappa}$, it can also be described as
\begin{equation}
\mathsf{U}_{\kappa}\left(\mathsf{E}\right)=\mathsf{D}_{\kappa}\left(\mathsf{E}\right)^{*}=\left\{ \mathbb{A}^{*}:\mathbb{A}\in\mathsf{D}_{\kappa}\left(\mathsf{E}\right)\right\} ,\label{uek}
\end{equation}
and it is equipped with an orthocomplemented complete lattice structure
\[
\left(\mathsf{U}_{\kappa}\left(\mathsf{E}\right),\leq,\amalg,\wedge,*,\hat{\mathbf{0}},\hat{\mathbf{1}}\right),
\]
with supremum
\begin{equation}
\coprod S=\left(\bigvee S\right)^{**},\quad\forall S\subseteq\mathsf{U}_{\kappa}\left(\mathsf{E}\right).\label{defs}
\end{equation}
\end{thm}

\begin{proof}
We have seen in Eq. \eqref{assa} that $A^{**}=A$ for all $A\in\mathsf{E}^{\kappa}$,
so $\mathsf{E}^{\kappa}\subseteq\mathsf{U}_{\kappa}\left(\mathsf{E}\right)$. 

To show \eqref{uek}, note first that, since for all $\mathbb{A}\in\mathsf{U}_{\kappa}\left(\mathsf{E}\right)$
we have that $\mathbb{A}=\mathbb{A}^{**}=\left(\mathbb{A}^{*}\right)^{*}$,
then $\mathsf{U}_{\kappa}\left(\mathsf{E}\right)\subseteq\mathsf{D}_{\kappa}\left(\mathsf{E}\right)^{*}$.
Secondly, since for all $\mathbb{A}\in\mathsf{D}_{\kappa}\left(\mathsf{E}\right)$
we have that $\mathbb{A}^{*}=\mathbb{A}^{***}=\left(\mathbb{A}^{*}\right)^{**}$,
the other inclusion follows.

Given $\mathbb{A}\in\mathsf{U}_{\kappa}\left(\mathsf{E}\right)\subseteq\mathsf{D}_{\kappa}\left(\mathsf{E}\right)$,
from \eqref{uek} we have that $\mathbb{A}^{*}\in\mathsf{U}_{\kappa}\left(\mathsf{E}\right)$.
Then, $*$ defines a function from $\mathsf{U}_{\kappa}\left(\mathsf{E}\right)$
to $\mathsf{U}_{\kappa}\left(\mathsf{E}\right)$, which is order-reversing
(see Proposition \ref{orev}) and an involution. To show that $*$
is an orthocomplementation on $\mathsf{U}_{\kappa}\left(\mathsf{E}\right)$,
it rests to prove that $\mathbb{A}^{*}$ is a complement of $\mathbb{A}$,
i.e.
\begin{equation}
\mathbb{A}\wedge\mathbb{A}^{*}=\hat{\mathbf{0}}\quad\textrm{and}\quad\mathbb{A}\amalg\mathbb{A}^{*}=\hat{\mathbf{1}},\quad\forall\mathbb{A}\in\mathsf{U}_{\kappa}\left(\mathsf{E}\right).\label{scomp}
\end{equation}
To show the first identity, consider an element $A=\left(a_{1},\ldots,a_{n},\ldots\right)\in\mathsf{E}^{\kappa}$.
Since 
\[
{\textstyle A^{*}=\bigvee_{n\in\kappa}\left(\mathbf{1},\ldots,\mathbf{1},a_{n}^{c},\mathbf{1},\ldots\right)}
\]
and
\[
A\wedge\left(\mathbf{1},\ldots,\mathbf{1},a_{n}^{c},\mathbf{1},\ldots\right)=\left(a_{1},\ldots,a_{n-1},a_{n}\cap a_{n}^{c},a_{n+1},\ldots\right)=\hat{\mathbf{0}},\quad\forall n\in\kappa,
\]
then (using distributivity)
\[
{\textstyle A\wedge A^{*}=\bigvee_{n\in\kappa}A\wedge\left(\mathbf{1},\ldots,\mathbf{1},a_{n}^{c},\mathbf{1},\ldots\right)=\hat{\mathbf{0}}.}
\]
Now, consider an arbitrary element $\mathbb{A}=\bigvee_{i\in I}A^{i}$
of $\mathsf{U}_{\kappa}\left(\mathsf{E}\right)$. The last result
ensures that $A^{i}\wedge\left(A^{i}\right)^{*}=\hat{\mathbf{0}}$,
and consequently 
\[
{\textstyle A^{i}\wedge\bigwedge_{l\in I}\left(A^{l}\right)^{*}=\hat{\mathbf{0}},\quad\forall i\in I.}
\]
Then
\[
{\textstyle \mathbb{A}\wedge\mathbb{A}^{*}=\left(\bigvee_{i\in I}A^{i}\right)\wedge\left(\bigwedge_{l\in I}\left(A^{l}\right)^{*}\right)=\bigvee_{i\in I}\left(A^{i}\wedge\bigwedge_{l\in I}\left(A^{l}\right)^{*}\right)}=\hat{\mathbf{0}}.
\]
That is to say, the first identity of \eqref{scomp} holds. From that,
and using Eqs. \eqref{dop3}, \eqref{dop4} and \eqref{defs}, we
have
\[
\mathbb{A}\amalg\mathbb{A}^{*}=\left(\mathbb{A}\vee\mathbb{A}^{*}\right)^{**}=\left(\mathbb{A}^{*}\wedge\mathbb{A}^{**}\right)^{*}=\left(\mathbb{A}^{*}\wedge\mathbb{A}\right)^{*}=\hat{\mathbf{0}}^{*}=\hat{\mathbf{1}},
\]
what shows the second identity of \eqref{scomp}. 
\end{proof}
We have shown in Proposition \ref{mml} that, for any family $\left\{ \mathbb{A}_{j}:j\in J\right\} \subseteq\mathsf{D}_{\kappa}\left(\mathsf{E}\right)$,
we have \eqref{dop4}. Inside $\mathsf{U}_{\kappa}\left(\mathsf{E}\right)$,
we have also its dual.
\begin{prop}
Given a subset $S=\left\{ \mathbb{A}_{j}:j\in J\right\} \subseteq\mathsf{U}_{\kappa}\left(\mathsf{E}\right)$,
the \textbf{infinite De Morgan's laws} hold, i.e.
\[
{\textstyle \left(\coprod_{j\in J}\mathbb{A}_{j}\right)^{*}=\bigwedge_{j\in J}\mathbb{A}_{j}^{*}\quad\textrm{and}\quad\left(\bigwedge_{j\in J}\mathbb{A}_{j}\right)^{*}=\coprod_{j\in J}\mathbb{A}_{j}^{*}.}
\]
\end{prop}

\begin{proof}
Let us prove the first identity. Since $\coprod_{j\in J}\mathbb{A}_{j}=\left(\bigvee_{j\in J}\mathbb{A}_{j}\right)^{**}$
(see \eqref{defs}), then, using Eq. \eqref{dop4} and \eqref{31},
\[
{\textstyle \left(\coprod_{j\in J}\mathbb{A}_{j}\right)^{*}=\left(\bigvee_{j\in J}\mathbb{A}_{j}\right)^{***}=\left(\bigvee_{j\in J}\mathbb{A}_{j}\right)^{*}=\bigwedge_{j\in J}\mathbb{A}_{j}^{*}.}
\]
For the second identity, note that
\[
{\textstyle \bigwedge_{j\in J}\mathbb{A}_{j}=\bigwedge_{j\in J}\mathbb{A}_{j}^{**}=\left(\bigvee_{j\in J}\mathbb{A}_{j}^{*}\right)^{*}}
\]
(where we have used \eqref{dop4} again). Then, applying $*$ member
to member in above equation,
\[
{\textstyle \left(\bigwedge_{j\in J}\mathbb{A}_{j}\right)^{*}=\left(\bigvee_{j\in J}\mathbb{A}_{j}^{*}\right)^{**}=\coprod_{j\in J}\mathbb{A}_{j}^{*},}
\]
what finishes the proof.
\end{proof}
\begin{rem*}
We know that an arbitrary element of $\mathsf{D}_{\kappa}\left(\mathsf{E}\right)$
can be written as ${\textstyle \bigvee_{i\in I}}A^{i}$, with $I\neq\emptyset$
and $A^{i}\in\mathsf{E}^{\kappa}$ for all $i\in I$. Then, the elements
of $\mathsf{U}_{\kappa}\left(\mathsf{E}\right)$ are given by those
symbols ${\textstyle \bigvee_{i\in I}}A^{i}$ such that
\[
{\textstyle \bigvee_{i\in I}}A^{i}=\left({\textstyle \bigvee_{i\in I}}A^{i}\right)^{**},
\]
i.e. those satisfying
\[
{\textstyle \bigvee_{i\in I}}A^{i}={\textstyle \coprod_{i\in I}}A^{i}.
\]
Thus, an arbitrary element of $\mathsf{U}_{\kappa}\left(\mathsf{E}\right)$
can be written as
\begin{equation}
{\textstyle \coprod_{i\in I}}A^{i},\quad\textrm{with}\quad I\neq\emptyset\quad\textrm{and}\quad A_{i}\in\mathsf{E}^{\kappa},\;\forall i\in I.\label{grep}
\end{equation}
In other words (compare to \eqref{EN} and \eqref{EBN}), 
\[
\mathsf{U}_{\kappa}\left(\mathsf{E}\right)=\left\{ {\textstyle \coprod_{i\in I}}A^{i}\;:\;A_{i}\in\mathsf{E}^{\kappa}\right\} .
\]
\end{rem*}
Recall that we are identifying $\mathsf{E}^{\kappa}$ with its image
$\varphi_{\kappa}\left(\mathsf{E}^{\kappa}\right)\subseteq\mathsf{D}_{\kappa}\left(\mathsf{E}\right)$
\textit{via} the order embedding $\varphi_{\kappa}$ (see Eq. \eqref{emb}).
So, the inclusion $\mathsf{E}^{\kappa}\subseteq\mathsf{U}_{\kappa}\left(\mathsf{E}\right)$
we talked about in Theorem \ref{pop} gives rise to the embedding
\[
\hat{\varphi}_{\kappa}:\mathsf{E}^{\kappa}\hookrightarrow\mathsf{U}_{\kappa}\left(\mathsf{E}\right),
\]
 obtained by the co-restriction of $\varphi_{\kappa}$ to $\mathsf{U}_{\kappa}\left(\mathsf{E}\right)$.
In the $\kappa=\left[1\right]$ case, we have the following result. 
\begin{thm}
\label{uee} $\mathsf{U}_{\left[1\right]}\left(\mathsf{E}\right)$
and $\mathsf{E}$ are isomorphic orthocomplemented complete lattices
via $\hat{\varphi}_{\left[1\right]}$.
\end{thm}

\begin{proof}
Let us show that the embedding 
\[
\hat{\varphi}_{\left[1\right]}:a\in\mathsf{E}\mapsto a\in\mathsf{U}_{\left[1\right]}\left(\mathsf{E}\right)
\]
is an isomorphism of orthocomplemented lattices. Since (by the very
definition of $*$) $a^{*}=a^{c}$, then $\hat{\varphi}_{\left[1\right]}$
respects the involved orthocomplementations. To end the proof, it
would be enough to show that $\hat{\varphi}_{\left[1\right]}$ is
surjective. According to the last remark (see Eq. \eqref{grep}),
the elements of $\mathsf{U}_{\left[1\right]}\left(\mathsf{E}\right)$
can be written as
\[
{\textstyle \coprod_{i\in I}}a^{i},\quad\textrm{with}\quad I\neq\emptyset\quad\textrm{and}\quad a_{i}\in\mathsf{E},\;\forall i\in I.
\]
But
\[
\begin{array}{lll}
{\textstyle \coprod_{i\in I}}a^{i} & = & \left({\textstyle \bigvee_{i\in I}}a^{i}\right)^{**}=\left({\textstyle \bigwedge_{i\in I}}\left(a^{i}\right)^{*}\right)^{*}=\left({\textstyle \bigwedge_{i\in I}}\left(a^{i}\right)^{c}\right)^{*}=\left[{\textstyle \bigcap_{i\in I}}\left(a^{i}\right)^{c}\right]^{*}\\
\\
 & = & \left({\textstyle \bigcap_{i\in I}}\left(a^{i}\right)^{c}\right)^{c}={\textstyle \bigcup_{i\in I}}a^{i}=\hat{\varphi}_{\left[1\right]}\left({\textstyle \bigcup_{i\in I}}a^{i}\right),
\end{array}
\]
as we wanted to show.
\end{proof}
So far, we have associated to every countable cardinal $\kappa$ and
every repeatable experiment, with event space $\left(\mathsf{E},\subseteq,\cup,\cap,^{c},\mathbf{0},\mathbf{1}\right)$,
the bounded lattice $\left(\mathsf{E}^{\kappa},\subseteq,\cup,\cap,\hat{\mathbf{0}},\hat{\mathbf{1}}\right)$
(with the component-wise operations), the logic $\left(\mathsf{U}_{\kappa}\left(\mathsf{E}\right),\leq,\amalg,\wedge,*,\hat{\mathbf{0}},\hat{\mathbf{1}}\right)$
and the order-embeddings $\hat{\varphi}_{\kappa}:\mathsf{E}^{\kappa}\hookrightarrow\mathsf{U}_{\kappa}\left(\mathsf{E}\right)$.
This has been done in such a way that, when $\kappa=\left[1\right]$,
we have that $\mathsf{U}_{\left[1\right]}\left(\mathsf{E}\right)\simeq\mathsf{E}$
(see Theorem \ref{uee}). Each element $A=\left(a_{1},\ldots,a_{n},\ldots\right)\in\mathsf{E}^{\kappa}\subseteq\mathsf{U}_{\kappa}\left(\mathsf{E}\right)$
(\textit{via} usual identifications) represents the event of the $\kappa$-times
repeated experiment: ``$a_{n}$ occurs in the $n$-th repetition
of the original experiment, for all $n\in\kappa$''. And an arbitrary
element ${\textstyle \coprod_{i\in I}}A^{i}$ of $\mathsf{U}_{\kappa}\left(\mathsf{E}\right)$
(see Eq. \eqref{grep}) represents the event: ``$A^{i}\in\mathsf{E}^{\kappa}$
occurs for some $i$.'' The operations $\amalg$, $\wedge$ and $*$
represent the OR, AND and NOT logic operations, respectively. 

We think that any reasonable definition of event space for repeated
experiments, related to a general logic $\mathsf{E}$, must be ``a
quotient of'' $\mathsf{U}_{\kappa}\left(\mathsf{E}\right)$. More
precisely, we shall consider the next definition. 
\begin{defn}
Given a repeatable experiment with event space $\mathsf{E}$, we shall
call $\mathsf{U}_{\kappa}\left(\mathsf{E}\right)$ the \textbf{universal
logic of the $\kappa$-times repeated experiment}, or simply the $\kappa$-\textbf{th}
\textbf{universal logic} for $\mathsf{E}$. And we shall call \textbf{concrete
event space of the} \textbf{$\kappa$-times repeated experiment} to
every logic $\mathsf{F}_{\kappa}$ for which there exists an epimorphism
of orthocomplemented lattices
\[
\mathsf{U}_{\kappa}\left(\mathsf{E}\right)\twoheadrightarrow\mathsf{F}_{\kappa}.
\]
For $\kappa=\mathbb{N}$, we shall write $\mathsf{U}_{\mathbb{N}}\left(\mathsf{E}\right)=\mathsf{U}\left(\mathsf{E}\right)$.
\end{defn}

In Section \ref{cle}, given a classical experiment with classical
logic $\mathsf{E}$ (a complete Boolean algebra), we have seen that
a good candidate for the event space of the \textbf{$\kappa$}-times
repeated experiment is the complete algebra generated by $\mathsf{E}^{\kappa}$,
i.e. the (concrete) logic $\left\langle \mathsf{E}^{\kappa}\right\rangle $.
We shall show in the next Section that the assignment
\[
{\textstyle \coprod_{i\in I}}A^{i}\in\mathsf{U}_{\kappa}\left(\mathsf{E}\right)\longmapsto{\textstyle \bigcup_{i\in I}}A^{i}\in\left\langle \mathsf{E}^{\kappa}\right\rangle 
\]
(where $\cup$ denotes the set union in the power-set of $\mathsf{E}^{\kappa}$)
defines a lattice epimorphism. This justifies the second part of above
definition.

\subsection{The epimorphism $\mathsf{U}_{\kappa}\left(\mathsf{E}\right)\twoheadrightarrow\left\langle \mathsf{E}^{\kappa}\right\rangle $}

Let $\mathsf{E}\subseteq2^{S}$ be a classical logic given by subsets
of $S$ and consider the related classical logic $\left\langle \mathsf{E}^{\kappa}\right\rangle \subseteq2^{S^{\kappa}}$
constructed in Section \ref{cle}. As in that section, by $\subseteq$,
$\cup$ and $\cap$ we shall denote the inclusion, union and intersection
of subsets in $S^{\kappa}$. Consider also the function
\[
\mathfrak{e}_{\kappa}:\left[{\textstyle \bigsqcup_{i\in I}}A^{i}\right]\in\mathsf{D}_{\kappa}\left(\mathsf{E}\right)\longmapsto{\textstyle \bigcup_{i\in I}}A^{i}\in\left\langle \mathsf{E}^{\kappa}\right\rangle .
\]
Note that
\begin{equation}
\mathfrak{e}_{\kappa}\left(\left[A\right]\right)=A,\quad\forall A\in\mathsf{E}^{\kappa}.\label{eka}
\end{equation}

\begin{thm}
Each function $\mathfrak{e}_{\kappa}$ defines, by restriction, an
epimorphism $\mathfrak{e}_{\kappa}:\mathsf{U}_{\kappa}\left(\mathsf{E}\right)\twoheadrightarrow\left\langle \mathsf{E}^{\kappa}\right\rangle $
of orthocomplemented complete lattices.
\end{thm}

\begin{proof}
It is easy to see that $\mathfrak{e}_{\kappa}:\mathsf{D}_{\kappa}\left(\mathsf{E}\right)\rightarrow\left\langle \mathsf{E}^{\kappa}\right\rangle $
is well-defined and surjective (since $\left\langle \mathsf{E}^{\kappa}\right\rangle $
is given precisely by the unions of the elements in $\mathsf{E}^{\kappa}$).
On the other hand, since (see Eq. \eqref{dop})
\[
{\textstyle \left[\bigsqcup_{i\in I}A^{i}\right]^{*}=\left[{\textstyle \bigsqcup_{f\in F}}\left(\bigcap_{i\in f^{-1}\left(1\right)}\left(a_{1}^{i}\right)^{c},\ldots,\bigcap_{i\in f^{-1}\left(n\right)}\left(a_{n}^{i}\right)^{c},\ldots\right)\right]},
\]
then (according to \eqref{fe2} and \eqref{fen})
\[
{\textstyle {\textstyle \mathfrak{e}_{\kappa}\left(\left[\bigsqcup_{i\in I}A^{i}\right]^{*}\right)}={\textstyle \bigcup_{f\in F}}\left(\bigcap_{i\in f^{-1}\left(1\right)}\left(a_{1}^{i}\right)^{c},\ldots,\bigcap_{i\in f^{-1}\left(n\right)}\left(a_{n}^{i}\right)^{c},\ldots\right)=\left(\bigcup_{i\in I}A^{i}\right)^{c}},
\]
i.e. 
\begin{equation}
{\textstyle \mathfrak{e}_{\kappa}\left(\mathbb{A}^{*}\right)}={\textstyle \left(\mathfrak{e}_{\kappa}\left(\mathbb{A}\right)\right)^{c},\quad\forall\mathbb{A}\in\mathsf{D}_{\kappa}\left(\mathsf{E}\right)}.\label{eks}
\end{equation}
So, using \eqref{uek} and the fact that $^{c}$ is an involution
on $\left\langle \mathsf{E}^{\kappa}\right\rangle $ (and, as a consequence,
any element of $\left\langle \mathsf{E}^{\kappa}\right\rangle $ is
the complement of another element of $\left\langle \mathsf{E}^{\kappa}\right\rangle $),
we have that $\mathfrak{e}_{\kappa}$ restricted $\mathsf{U}_{\kappa}\left(\mathsf{E}\right)$
is also surjective.

So far, we have a surjective map $\mathsf{U}_{\kappa}\left(\mathsf{E}\right)\twoheadrightarrow\left\langle \mathsf{E}^{\kappa}\right\rangle $
that respects orthocomplementation. It rests to show that it is a
join homomorphism.\footnote{The meet homomorphism property will follow from the De Morgan laws. }
Given $\coprod_{i\in I}A^{i}\in\mathsf{U}_{\kappa}\left(\mathsf{E}\right)$,
and using the last equation,
\[
{\textstyle \mathfrak{e}_{\kappa}\left(\coprod_{i\in I}A^{i}\right)=\mathfrak{e}_{\kappa}\left(\left(\bigvee_{i\in I}A^{i}\right)^{**}\right)=\left(\mathfrak{e}_{\kappa}\left(\bigvee_{i\in I}A^{i}\right)\right)^{cc}=\mathfrak{e}_{\kappa}\left(\bigvee_{i\in I}A^{i}\right)={\textstyle \bigcup_{i\in I}}A^{i}}.
\]
And given a subset $S=\left\{ \mathbb{A}_{j}\in\mathsf{U}_{\kappa}\left(\mathsf{E}\right):j\in J\right\} $,
with $\mathbb{A}_{j}=\coprod_{i\in I_{j}}A_{j}^{i}$,
\begin{equation}
\begin{array}{lll}
\mathfrak{e}_{\kappa}\left(\coprod_{j\in J}\mathbb{A}_{j}\right) & = & \mathfrak{e}_{\kappa}\left(\coprod_{j\in J}\coprod_{i\in I_{j}}A_{j}^{i}\right)=\bigcup_{j\in J}\bigcup_{i\in I_{j}}A_{j}^{i}=\bigcup_{j\in J}\mathfrak{e}_{\kappa}\left(\coprod_{i\in I_{j}}A_{j}^{i}\right)\\
\\
 & = & \bigcup_{j\in J}\mathfrak{e}_{\kappa}\left(\mathbb{A}_{j}\right),
\end{array}\label{eku}
\end{equation}
what ends our proof. 
\end{proof}

\section{Distributivity of $\mathsf{U}_{\kappa}\left(\mathsf{E}\right)$}

\label{doue} In this section, we shall show the following theorem.
\begin{thm}
If $\mathsf{E}$ is distributive, then each universal logic $\mathsf{U}_{\kappa}\left(\mathsf{E}\right)$
is distributive, for all $\kappa$. Reciprocally, if $\mathsf{U}_{\kappa}\left(\mathsf{E}\right)$
is distributive for some $\kappa$, then $\mathsf{E}$ is distributive.
\end{thm}

To do that, we need the next concept. A bounded lattice $\left(\mathsf{L},\leq,\vee,\wedge,0,1\right)$
is said to be a $p$-\textbf{algebra} if there exists a unary operator
\[
\star:\mathsf{L}\rightarrow\mathsf{L}:x\mapsto x^{\star},
\]
 called pseudo-complement, such that:
\begin{itemize}
\item $x\wedge\left(x\wedge y\right)^{\star}=x\wedge y^{\star}$;
\item $x\wedge0^{\star}=x$;
\item $0^{\star\star}=0$.
\end{itemize}
(If such an operator exists, then it is unique). Suppose that $\mathsf{L}$
is complete. What is important for us is the following fact (see Ref.
\cite{blyth} for details): on the sub-poset 
\[
\mathsf{S}\left(\mathsf{L}\right)\coloneqq\left\{ x\in\mathsf{L}:x^{\star\star}=x\right\} ,
\]
the \textit{skeleton} of $\mathsf{L}$, we have an orthocomplemented
complete lattice structure with the same bottom and top elements,
the same infimum, i.e.
\[
{\textstyle \bigwedge_{\mathsf{S}\left(\mathsf{L}\right)}S=\bigwedge S,}\quad\forall S\subseteq\mathsf{S}\left(\mathsf{L}\right),
\]
supremum given by

\[
{\textstyle \coprod S\coloneqq\bigvee_{\mathsf{S}\left(\mathsf{L}\right)}S=\left(\bigvee_{L}S\right)^{\star\star},\quad\forall S\subseteq\mathsf{S}\left(\mathsf{L}\right),}
\]
and orthocomplementation given by the restriction of $\star$ to $\mathsf{S}\left(\mathsf{L}\right)$
(this is because $\star\star$ is a closure for $\mathsf{L}$). And
what is even more important is that, such a structure:
\[
\left(\mathsf{S}\left(\mathsf{L}\right)\leq,\amalg,\wedge,\star,0,1\right),
\]
is a distributive lattice, and it satisfies the infinite distributive
law
\begin{equation}
{\textstyle z\wedge\left(\coprod_{i\in I}x_{i}\right)=\coprod_{i\in I}\left(z\wedge x_{i}\right)}.\label{idl}
\end{equation}
(The proof given in \cite{blyth} is for finite sets $I$, but it
can be easily extended to arbitrary sets in the complete case).

\subsection{Proof of the direct result}

Fix a cardinal $\kappa$ and consider the distributive lattice $\mathsf{D}_{\kappa}\left(\mathsf{E}\right)$. 
\begin{prop}
\label{prop23} If $\mathsf{E}$ is distributive, then $*$ is a pseudo-complement
for $\mathsf{D}_{\kappa}\left(\mathsf{E}\right)$.
\end{prop}

\begin{proof}
It is enough to show that 
\begin{equation}
\mathbb{A}\wedge\left(\mathbb{A}\wedge\mathbb{B}\right)^{*}=\mathbb{A}\wedge\mathbb{B}^{*},\quad\forall\mathbb{A},\mathbb{B}\in\mathsf{D}_{\kappa}\left(\mathsf{E}\right),\label{equ}
\end{equation}
the first of the three properties that makes $\ast$ a pseudo-complement.
(The other two properties follows from the fact that $\hat{\mathbf{0}}^{*}=\hat{\mathbf{1}}$
and $\hat{\mathbf{1}}^{*}=\hat{\mathbf{0}}$). For $\mathbb{A}=\bigvee_{i\in I}A^{i}$
and $\mathbb{B}=\bigvee_{j\in J}B^{j}$, the l.h.s. of above equation
says that
\[
{\textstyle \mathbb{A}\wedge\left(\mathbb{A}\wedge\mathbb{B}\right)^{*}=\bigvee_{i\in I}\bigwedge_{l\in I}\bigwedge_{j\in J}A^{i}\wedge\left(A^{l}\cap B^{j}\right)^{*},}
\]
and the l.h.s.
\[
{\textstyle {\textstyle \mathbb{A}\wedge\mathbb{B}^{*}=\bigvee_{i\in I}\bigwedge_{j\in J}A^{i}\wedge\left(B^{j}\right)^{*}}.}
\]
On the one hand, since
\begin{equation}
\begin{array}{lll}
A^{i}\wedge\left(A^{l}\cap B^{j}\right)^{*} & = & A^{i}\wedge\bigvee_{n\in\kappa}\left(\mathbf{1},\ldots,\mathbf{1},\left(a_{n}^{l}\cap b_{n}^{j}\right)^{c},\mathbf{1},\ldots\right)\\
\\
 & = & \bigvee_{n\in\kappa}A^{i}\cap\left(\mathbf{1},\ldots,\mathbf{1},\left(a_{n}^{l}\right)^{c}\cup\left(b_{n}^{j}\right)^{c},\mathbf{1},\ldots\right)\\
\\
 & = & \bigvee_{n\in\kappa}\left(a_{1}^{i},\ldots,a_{n-1}^{i},a_{n}^{i}\cap\left(\left(a_{n}^{l}\right)^{c}\cup\left(b_{n}^{j}\right)^{c}\right),a_{n+1}^{i},\ldots\right),
\end{array}\label{l0i}
\end{equation}
and
\[
a_{n}^{i}\cap\left(b_{n}^{j}\right)^{c}\subseteq a_{n}^{i}\cap\left(\left(a_{n}^{l}\right)^{c}\cup\left(b_{n}^{j}\right)^{c}\right),
\]
then
\[
{\textstyle A^{i}\wedge\left(B^{j}\right)^{*}=\bigvee_{n\in\kappa}\left(a_{1}^{i},\ldots,a_{n-1}^{i},a_{n}^{i}\cap\left(b_{n}^{j}\right)^{c},a_{n+1}^{i},\ldots\right)\leq A^{i}\wedge\left(A^{l}\cap B^{j}\right)^{*}.}
\]
This implies that
\[
{\textstyle \bigvee_{i\in I}\bigwedge_{j\in J}A^{i}\wedge\left(B^{j}\right)^{*}\leq\bigvee_{i\in I}\bigwedge_{l\in I}\bigwedge_{j\in J}A^{i}\wedge\left(A^{l}\cap B^{j}\right)^{*},}
\]
i.e.
\[
\mathbb{A}\wedge\mathbb{B}^{*}\leq\mathbb{A}\wedge\left(\mathbb{A}\wedge\mathbb{B}\right)^{*}.
\]
On the other hand, for all $l_{0}\in I$, it is clear that
\[
{\textstyle \bigwedge_{l\in I}\bigwedge_{j\in J}A^{i}\wedge\left(A^{l}\cap B^{j}\right)^{*}\leq\bigwedge_{j\in J}A^{i}\wedge\left(A^{l_{0}}\cap B^{j}\right)^{*}.}
\]
Taking $l_{0}=i$, and according to Eq. \eqref{l0i}, we have that
\[
\begin{array}{lll}
A^{i}\wedge\left(A^{i}\cap B^{j}\right)^{*} & = & \bigvee_{n\in\kappa}\left(a_{1}^{i},\ldots,a_{n-1}^{i},a_{n}^{i}\cap\left(\left(a_{n}^{i}\right)^{c}\cup\left(b_{n}^{j}\right)^{c}\right),a_{n+1}^{i},\ldots\right)\\
\\
 & = & \bigvee_{n\in\kappa}\left(a_{1}^{i},\ldots,a_{n-1}^{i},a_{n}^{i}\cap\left(b_{n}^{j}\right)^{c},a_{n+1}^{i},\ldots\right)\\
\\
 & = & A^{i}\wedge\left(B^{j}\right)^{*},
\end{array}
\]
where we have used the distributivity of $\mathsf{E}$, what implies
that
\begin{equation}
a_{n}^{i}\cap\left(\left(a_{n}^{i}\right)^{c}\cup\left(b_{n}^{j}\right)^{c}\right)=\left(a_{n}^{i}\cap\left(a_{n}^{i}\right)^{c}\right)\cup\left(a_{n}^{i}\cap\left(b_{n}^{j}\right)^{c}\right)=a_{n}^{i}\cap\left(b_{n}^{j}\right)^{c}.\label{l0j}
\end{equation}
So, 
\[
{\textstyle \bigwedge_{l\in I}\bigwedge_{j\in J}A^{i}\wedge\left(A^{l}\cap B^{j}\right)^{*}\leq\bigwedge_{j\in J}A^{i}\wedge\left(B^{j}\right)^{*},}
\]
from which
\[
\mathbb{A}\wedge\left(\mathbb{A}\wedge\mathbb{B}\right)^{*}\leq\mathbb{A}\wedge\mathbb{B}^{*},
\]
and the equality \eqref{equ} follows.
\end{proof}
So, if $\mathsf{E}$ is distributive,
\[
\mathsf{S}\left(\mathsf{D}_{\kappa}\left(\mathsf{E}\right)\right)=\mathsf{D}_{\kappa}\left(\mathsf{E}\right)^{**}=\mathsf{U}_{\kappa}\left(\mathsf{E}\right)
\]
and the orthocomplemented complete lattice structures on $\mathsf{S}\left(\mathsf{D}_{\kappa}\left(\mathsf{E}\right)\right)$
and $\mathsf{U}_{\kappa}\left(\mathsf{E}\right)$ coincide. As a consequence,
$\mathsf{U}_{\kappa}\left(\mathsf{E}\right)$ is distributive. Moreover
(see Eq. \eqref{idl}), it satisfies the infinite distributive law
\begin{equation}
{\textstyle \mathbb{A}\wedge\left(\coprod_{j\in J}\mathbb{B}_{j}\right)=\coprod_{j\in J}\left(\mathbb{A}\wedge\mathbb{B}_{j}\right).}\label{idl2}
\end{equation}

\begin{rem*}
It is worth mentioning that the complete distributivity of $\mathsf{D}_{\kappa}\left(\mathsf{E}\right)$
implies the existence of a (unique) pseudo-complement $\star:\mathsf{D}_{\kappa}\left(\mathsf{E}\right)\rightarrow\mathsf{D}_{\kappa}\left(\mathsf{E}\right)$,
which is given by
\[
\mathbb{A}^{\star}=\bigvee\left\{ \mathbb{B}:\mathbb{A}\wedge\mathbb{B}=\hat{\mathbf{0}}\right\} .
\]
Above result simply says that $\star=*$ when $\mathsf{E}$ is distributive.
\end{rem*}

\subsection{Proof of the converse}

Consider elements of the form $\left(x,\hat{\mathbf{1}}\right)\in\mathsf{U}_{\kappa}\left(\mathsf{E}\right)$,
with $x\in\mathsf{E}$. Then, for such elements, 
\[
\begin{array}{lll}
\left(x,\hat{\mathbf{1}}\right)\amalg\left(y,\hat{\mathbf{1}}\right) & = & \left(\left(x,\hat{\mathbf{1}}\right)\vee\left(y,\hat{\mathbf{1}}\right)\right)^{**}=\left(\left(x,\hat{\mathbf{1}}\right)^{*}\wedge\left(y,\hat{\mathbf{1}}\right)^{*}\right)^{*}\\
\\
 & = & \left(\left(x^{c},\hat{\mathbf{1}}\right)\wedge\left(y^{c},\hat{\mathbf{1}}\right)\right)^{*}=\left(x^{c}\cap y^{c},\hat{\mathbf{1}}\right)^{*}=\left(x\cup y,\hat{\mathbf{1}}\right).
\end{array}
\]
So, given $a,b,c\in\mathsf{E}$, 
\[
\left(\left(a,\hat{\mathbf{1}}\right)\amalg\left(b,\hat{\mathbf{1}}\right)\right)\wedge\left(c,\hat{\mathbf{1}}\right)=\left(a\cup b,\hat{\mathbf{1}}\right)\cap\left(c,\hat{\mathbf{1}}\right)=\left(\left(a\cup b\right)\cap c,\hat{\mathbf{1}}\right)
\]
and
\[
\left(\left(a,\hat{\mathbf{1}}\right)\wedge\left(c,\hat{\mathbf{1}}\right)\right)\amalg\left(\left(b,\hat{\mathbf{1}}\right)\wedge\left(c,\hat{\mathbf{1}}\right)\right)=\left(a\cap c,\hat{\mathbf{1}}\right)\amalg\left(b\cap c,\hat{\mathbf{1}}\right)=\left(\left(a\cap c\right)\cup\left(b\cap c\right),\hat{\mathbf{1}}\right).
\]
As a consequence, if we assume that $\mathsf{U}_{\kappa}\left(\mathsf{E}\right)$
is distributive, then the same is true for $\mathsf{E}$.

\section{The \textit{event space completion}}

\label{eec} At the beginning of the paper, we said that we would
concentrate exclusively on event spaces given by orthocomplemented
complete lattices. But, what happens if, for some reason, the event
space of an experiment is \textit{naturally} given by an orthocomplemented
lattice $\mathsf{L}$ which is not necessarily complete? This is the
case, for instance, of the classical statistical mechanical systems,
whose event spaces are (non-necessarily complete) $\sigma$-algebras.
We shall show in this Section that, in such cases, we can replace
$\mathsf{L}$ by a complete orthocomplemented lattice $\mathsf{E}_{\mathsf{L}}$
that embeds $\mathsf{L}$, in such a way that, in some sense, no new
\textit{artificial} events must be added to $\mathsf{L}$. Moreover,
$\mathsf{E}_{\mathsf{L}}$ is distributive if and only if $\mathsf{L}$
is distributive. (The construction is similar to that of the universal
logic, but in the particular case in which $\kappa=\left[1\right]$). 

\bigskip{}

Consider an event space given by a orthocomplemented lattice $\left(\mathsf{L},\subseteq,\cup,\cap,^{c},\mathbf{0},\mathbf{1}\right)$.
Consider also the set of terms 
\[
\mathsf{T}\left(\mathsf{L}\right)=\left\{ {\textstyle \bigsqcup_{i\in I}}a^{i}\;:\;\left|I\right|\leq2^{\left|\mathsf{L}\right|},\quad a^{i}\in\mathsf{L}\right\} 
\]
(compare to \eqref{fts} for $\kappa=\left[1\right]$), the equivalence
relation $\backsim$ of Section \ref{tap} and the quotient
\[
\mathsf{D}\left(\mathsf{L}\right)\coloneqq\left.\mathsf{T}\left(\mathsf{L}\right)\right/\backsim.
\]
For such an equivalence, for all $a,b\in\mathsf{L}$, we have again
that $\left[a\right]=\left[b\right]$, i.e. $a\backsim b$, if and
only if $a=b$. So, we can see $\mathsf{L}$ as a subset of $\mathsf{D}\left(\mathsf{L}\right)$
via the injective function $a\mapsto\left[a\right]$. Following similar
steps as in Section \ref{tcdis}, it can be shown that a completely
distributive lattice structure
\[
\left(\mathsf{D}\left(\mathsf{L}\right),\leq,\vee,\wedge,\mathbf{0},\mathbf{1}\right)
\]
can be defined on $\mathsf{D}\left(\mathsf{L}\right)$. Also, the
elements of $\mathsf{D}\left(\mathsf{L}\right)$ can be described
as:
\[
{\textstyle \bigvee_{i\in I}a^{i},\quad\textrm{with}\quad I\neq\emptyset\quad\textrm{and}\quad a_{i}\in\mathsf{L},\;\forall i\in I.}
\]

\begin{rem}
\label{rvu} The only difference with the results obtained in Section
\ref{tcdis} is that, in the formula for the infimum (see \eqref{finf}),
we must replace $\bigcap$ by $\bigwedge$, unless the involved elements
are finite or $\mathsf{L}$ is already complete. So, on $\mathsf{D}\left(\mathsf{L}\right)$,
we have in general that:
\begin{equation}
\bigwedge S=\bigcap S,\quad\forall S\subseteq\mathsf{L}\subseteq\mathsf{D}\left(\mathsf{L}\right),\quad\textrm{if}\;S\;\textrm{is finite}.\label{iinc}
\end{equation}
\end{rem}

Now, following similar steps as in Section \ref{unlo}, it can be
shown that the formula
\[
{\textstyle \left(\bigvee_{i\in I}a^{i}\right)^{*}=\bigwedge_{i\in I}\left(a^{i}\right)^{c}}
\]
defines a order-reversing function $*:\mathsf{D}\left(\mathsf{L}\right)\rightarrow\mathsf{D}\left(\mathsf{L}\right)$
such that 
\[
\left(\bigvee S\right)^{*}=\bigwedge S^{*},\quad\forall S\subseteq\mathsf{D}\left(\mathsf{L}\right).
\]
Moreover, $\mathbb{A}^{***}=\mathbb{A}^{*}$ and $\mathbb{A}\leq\mathbb{A}^{**}$
for all $\mathbb{A}\in\mathsf{D}\left(\mathsf{L}\right)$, what implies
that $**:\mathsf{D}\left(\mathsf{L}\right)\rightarrow\mathsf{D}\left(\mathsf{L}\right)$
is a closure operator. Consequently, as we saw in Theorem \ref{pop},
$\mathsf{D}\left(\mathsf{L}\right)^{**}$ is an orthocomplemented
complete lattice, with ortho-complement given by $*$. Also, the injective
function $a\in\mathsf{L}\mapsto\left[a\right]\in\mathsf{D}\left(\mathsf{L}\right)$
has its image inside $\mathsf{\mathsf{D}\left(\mathsf{L}\right)^{**}}$,
so we can see $\mathsf{L}$ as a subset of $\mathsf{\mathsf{D}\left(\mathsf{L}\right)^{**}}$.
All that gives rise to the following definition.
\begin{defn}
Given an orthocomplemented lattice $\left(\mathsf{L},\subseteq,\cup,\cap,^{c},\mathbf{0},\mathbf{1}\right)$,
we define the \textbf{event space completion }of $\mathsf{L}$ as
the orthocomplemented complete lattice (i.e. the logic)
\[
\left(\mathsf{E}_{\mathsf{L}},\leq,\amalg,\wedge,*,\mathbf{0},\mathbf{1}\right)
\]
with 
\begin{equation}
\mathsf{E}_{\mathsf{L}}=\mathsf{D}\left(\mathsf{L}\right)^{**}=\left\{ \mathbb{A}\in\mathsf{D}\left(\mathsf{L}\right):\mathbb{A}=\mathbb{A}^{**}\right\} =\left\{ \mathbb{A}^{*}:\mathbb{A}\in\mathsf{D}\left(\mathsf{L}\right)\right\} \label{el}
\end{equation}
and supremum
\[
\coprod S=\left(\bigvee S\right)^{**},\quad\forall S\subseteq\mathsf{E}_{\mathsf{L}}.
\]
\end{defn}

Again, the elements of $\mathsf{E}_{\mathsf{L}}$ are given by
\[
{\textstyle \coprod_{i\in I}a^{i},\quad\textrm{with}\quad I\neq\emptyset\quad\textrm{and}\quad a_{i}\in\mathsf{L},\;\forall i\in I.}
\]
In other words, $\mathsf{E}_{\mathsf{L}}$ contains the original events
of $\mathsf{L}$ together with new elements: $\coprod_{i\in I}a^{i}$.
The latter represent the propositions of the form: ``$a^{i}\in\mathsf{L}$
occurs for some $i\in I$''. Then, we can say that no new artificial
events were added to $\mathsf{L}$. Moreover, we have the following
important result. 
\begin{thm}
The inclusion $\mathsf{inc}:\mathsf{L}\hookrightarrow\mathsf{E}_{\mathsf{L}}$
is an injective homomorphism of orthocomplemented lattices, and it
is an isomorphism if and only if $\mathsf{L}$ is complete. Also,
if $\mathsf{L}$ and $\mathsf{L}'$ are isomorphic orthocomplemented
lattices, so are $\mathsf{E}_{\mathsf{L}}$ and $\mathsf{E}_{\mathsf{L'}}$. 
\end{thm}

\begin{proof}
Given $a\in\mathsf{L}$, we know that $a^{*}=a^{c}=\mathsf{inc}\left(a^{c}\right)$.
Also, given in addition $b\in\mathsf{L}$, from Eq. \eqref{iinc}
we have that
\begin{equation}
a\wedge b=a\cap b=\mathsf{inc}\left(a\cap b\right).\label{i}
\end{equation}
On the other hand, by the very definition of $*$ and by using the
De Morgan laws on $\mathsf{L}$,
\begin{equation}
a\amalg b=\left(a\vee b\right)^{**}=\left(a^{*}\wedge b^{*}\right)^{*}=\left(a^{c}\cap b^{c}\right)^{c}=a\cup b=\mathsf{inc}\left(a\cup b\right).\label{uu}
\end{equation}
All that proves that the inclusion $\mathsf{inc}:\mathsf{L}\hookrightarrow\mathsf{E}_{\mathsf{L}}$
is an injective homomorphism of orthocomplemented lattices. If $\mathsf{L}$
is complete (see Eq. \eqref{iinc} and Remark \ref{rvu}), then
\[
\bigwedge S=\bigcap S,\quad\forall S\subseteq\mathsf{L},
\]
so,
\[
{\textstyle \left(\coprod_{i\in I}a^{i}\right)^{*}=\bigwedge_{i\in I}\left(a^{i}\right)^{*}=\bigcap_{i\in I}\left(a^{i}\right)^{c}\in\mathsf{L}.}
\]
This implies that any element of $\mathsf{E}_{\mathsf{L}}$ is included
in $\mathsf{L}$, i.e. $\mathsf{L}=\mathsf{E}_{\mathsf{L}}$, and
as a consequence $\mathsf{inc}$ is the identity map, \textit{ipso-facto}
an isomorphism.

Suppose that $\chi:\mathsf{L}\rightarrow\mathsf{L}'$ is an isomorphism
of orthocomplemented lattices. It is easy to see that 
\[
\hat{\chi}:\left[{\textstyle \bigsqcup_{i\in I}}a^{i}\right]\in\mathsf{D}\left(\mathsf{L}\right)\longmapsto\left[{\textstyle \bigsqcup_{i\in I}}\chi\left(a^{i}\right)\right]\in\mathsf{D}\left(\mathsf{L}'\right)
\]
is a well-defined isomorphism of complete lattices. In particular,
\[
\hat{\chi}\left(\bigvee S\right)=\bigvee\hat{\chi}\left(S\right)\quad\textrm{and}\quad\hat{\chi}\left(\bigwedge S\right)=\bigwedge\hat{\chi}\left(S\right),\quad\forall S\subseteq\mathsf{D}\left(\mathsf{L}\right).
\]
Moreover,
\[
\begin{array}{lll}
\hat{\chi}\left(\left({\textstyle \bigvee_{i\in I}}a^{i}\right)^{*}\right) & = & \hat{\chi}\left({\textstyle \bigwedge_{i\in I}}\left(a^{i}\right)^{*}\right)={\textstyle \bigwedge_{i\in I}}\hat{\chi}\left(\left(a^{i}\right)^{*}\right)={\textstyle \bigwedge_{i\in I}}\chi\left(\left(a^{i}\right)^{c}\right)\\
\\
 & = & {\textstyle \bigwedge_{i\in I}}\left(\chi\left(a^{i}\right)\right)^{c}={\textstyle \bigwedge_{i\in I}}\left(\chi\left(a^{i}\right)\right)^{*}=\left({\textstyle \bigvee_{i\in I}}\chi\left(a^{i}\right)\right)^{*},
\end{array}
\]
what implies that $\hat{\chi}\left(\mathsf{E}_{\mathsf{L}}\right)\subseteq\mathsf{E}_{\mathsf{L}'}$
(see Eq. \eqref{el}) and that $\hat{\chi}$ defines an injective
homomorphism $\mathsf{E}_{\mathsf{L}}\hookrightarrow\mathsf{E}_{\mathsf{L}'}$
that respect $*$. The same can be said for the inverse of $\hat{\chi}$,
\[
\hat{\chi}^{-1}:\left[{\textstyle \bigsqcup_{i\in I}}c^{i}\right]\in\mathsf{D}\left(\mathsf{L}'\right)\longmapsto\left[{\textstyle \bigsqcup_{i\in I}}\chi^{-1}\left(c^{i}\right)\right]\in\mathsf{D}\left(\mathsf{L}\right),
\]
what ensures that $\hat{\chi}:\mathsf{E}_{\mathsf{L}}\rightarrow\mathsf{E}_{\mathsf{L}'}$
defines the isomorphism of orthocomplemented complete lattices we
are looking for.
\end{proof}
Regarding the distributivity, we have the following annunciated result.
\begin{thm}
$\mathsf{E}_{\mathsf{L}}$ is distributive if and only if $\mathsf{L}$
is distributive.
\end{thm}

\begin{proof}
Given $a,b,c\in\mathsf{L}$, using \eqref{i} and\eqref{uu} we have
that
\[
\left(a\amalg b\right)\wedge c=\left(a\cup b\right)\wedge c=\left(a\cup b\right)\cap c
\]
and
\[
\left(a\wedge c\right)\amalg\left(b\wedge c\right)=\left(a\cap c\right)\amalg\left(b\cap c\right)=\left(a\cap c\right)\cup\left(b\cap c\right).
\]
So, if $\mathsf{E}_{\mathsf{L}}$ is distributive, then so is $\mathsf{L}$.
To show the converse, we can prove that $*$ is a pseudo-complement
on $\mathsf{D}\left(\mathsf{L}\right)$. And to do that, it is enough
to repeat the proof of Proposition \ref{prop23} for the $\kappa=\left[1\right]$
case.
\end{proof}

\section{The tensor product presentation}

\label{tpp} In this section we shall extend the construction of the
universal logic to the case in which the event spaces change from
one repetition to another, and also to the case in which the cardinality
of $\kappa$ is arbitrary. This gives rise, in turn, to another way
of seeing the universal logic: as a tensor product of orthocomplemented
complete lattices.

\subsection{The logic ${\textstyle \bigotimes_{\alpha\in\kappa}}\mathsf{E}_{\alpha}$}

Consider a family $\left\{ \mathsf{E}_{\alpha}\right\} _{\alpha\in\kappa}$
of logics $\left(\mathsf{E}_{\alpha},\subseteq\cup,\cap,^{c},\mathbf{0},\mathbf{1}\right)$,
where $\kappa$ is an index set with arbitrary cardinality, and consider
the functions 
\[
A:\kappa\rightarrow{\textstyle \bigcup_{\alpha\in\kappa}}\mathsf{E}_{\alpha}\quad\textrm{such that}\quad A\left(\alpha\right)\in\mathsf{E}_{\alpha}.
\]
If $A\left(\alpha\right)=\mathbf{0}$ for some $\alpha$, then we
shall identify $A$ with the function with constant value $\mathbf{0}$,
which we shall indicate $\hat{\mathbf{0}}$. Also, let us indicate
$\hat{\mathbf{1}}$ the function with constant value $\mathbf{1}$.
Finally, denote by $\mathsf{F}_{\kappa}$ the set of above functions
(with the mentioned identification). It is clear that $\mathsf{F}_{\kappa}$
is a complete lattice with respect to the component-wise order: $A\subseteq B$
if and only if $A\left(\alpha\right)\subseteq B\left(\alpha\right)$,
$\forall\alpha\in\kappa$. Its bottom is $\hat{\mathbf{0}}$ and its
top is $\hat{\mathbf{1}}$. 
\begin{rem*}
Of course, if for instance $\kappa=\mathbb{N}$ and $\mathsf{E}_{\alpha}=\mathsf{E}$
for all $\alpha\in\kappa$, being $\mathsf{E}$ some given logic,
then $\mathsf{F}_{\kappa}=\mathsf{E}^{\mathbb{N}}$.
\end{rem*}
Consider now the set of terms 
\[
\mathsf{T}\left(\mathsf{F}_{\kappa}\right)=\left\{ {\textstyle \bigsqcup_{i\in I}}A^{i}\;:\;\left|I\right|\leq2^{\left|\mathsf{F}_{\kappa}\right|},\quad A^{i}\in\mathsf{F}_{\kappa}\right\} ,
\]
the equivalence relation $\backsim$ of Section \ref{tap}, and the
quotient
\[
\mathsf{D}\left(\mathsf{F}_{\kappa}\right)\coloneqq\left.\mathsf{T}\left(\mathsf{F}_{\kappa}\right)\right/\backsim.
\]
For such an equivalence, for all $A,B\in\mathsf{F}_{\kappa}$, we
have again that $\left[A\right]=\left[B\right]$, i.e. $A\backsim B$,
if and only if $A=B$. So, we can see $\mathsf{F}_{\kappa}$ as a
subset of $\mathsf{D}\left(\mathsf{F}_{\kappa}\right)$ via the injective
function 
\begin{equation}
\varphi:A\in\mathsf{\mathsf{F}_{\kappa}}\mapsto\left[A\right]\in\mathsf{D}\left(\mathsf{F}_{\kappa}\right).\label{fial}
\end{equation}
Following similar steps as in Section \ref{tcdis}, it can be shown
that a completely distributive lattice structure
\[
\left(\mathsf{D}\left(\mathsf{F}_{\kappa}\right),\leq,\vee,\wedge,\mathbf{0},\mathbf{1}\right)
\]
can be defined on $\mathsf{D}\left(\mathsf{F}_{\kappa}\right)$. Regarding
this structure, and that of $\mathsf{F}_{\kappa}$, the above injection
$\varphi$ is an order embedding. Also, the elements of $\mathsf{D}\left(\mathsf{F}_{\kappa}\right)$
can be described as:
\begin{equation}
{\textstyle \bigvee_{i\in I}A^{i},\quad\textrm{with}\quad I\neq\emptyset\quad\textrm{and}\quad A_{i}\in\mathsf{F}_{\kappa},\;\forall i\in I,}\label{cbd}
\end{equation}
where we are identifying $A$ and $\left[A\right]$ \textit{via} $\varphi$
(see \eqref{fial}). And the Eq. \eqref{finf} gives a formula for
the infimum in $\mathsf{D}\left(\mathsf{F}_{\kappa}\right)$. Moreover,
as in Section \ref{unlo}, it can be shown that the formula
\begin{equation}
{\textstyle \left(\bigvee_{i\in I}A^{i}\right)^{*}={\textstyle \bigvee_{f\in F}}\mathbb{A}_{f,I}},\quad\mathbb{A}_{f,I}\left(\alpha\right)\coloneqq{\textstyle \bigcap_{f\left(i\right)=\alpha}}\left(A^{i}\left(\alpha\right)\right)^{c},\label{fc}
\end{equation}
with $F\coloneqq\left\{ f:I\rightarrow\kappa\right\} $, defines a
order-reversing function $*:\mathsf{D}\left(\mathsf{F}_{\kappa}\right)\rightarrow\mathsf{D}\left(\mathsf{F}_{\kappa}\right)$
such that 
\[
\left(\bigvee S\right)^{*}=\bigwedge S^{*},\quad\forall S\subseteq\mathsf{D}\left(\mathsf{F}_{\kappa}\right).
\]
In particular, for $A\in\mathsf{F}_{\kappa}$,
\begin{equation}
A^{*}={\textstyle \bigvee_{\alpha\in\kappa}}A_{\alpha},\quad\textrm{with}\quad A_{\alpha}\left(\beta\right)=\begin{cases}
\left(A\left(\alpha\right)\right)^{c}, & \beta=\alpha,\\
\mathbf{1}, & \beta\neq\alpha.
\end{cases}\label{cf0}
\end{equation}
On the other hand, we can also see that $\mathbb{A}^{***}=\mathbb{A}^{*}$
and $\mathbb{A}\leq\mathbb{A}^{**}$ for all $\mathbb{A}\in\mathsf{D}\left(\mathsf{F}_{\kappa}\right)$,
what implies that $**:\mathsf{D}\left(\mathsf{F}_{\kappa}\right)\rightarrow\mathsf{D}\left(\mathsf{F}_{\kappa}\right)$
is a closure operator. Consequently, as we saw in Theorem \ref{pop},
$\mathsf{D}\left(\mathsf{F}_{\kappa}\right)^{**}$ is an orthocomplemented
complete lattice, with orthocomplementation given by $*$. Also, the
injective function $\varphi$ (see Eq. \eqref{fial}) has its image
inside $\mathsf{\mathsf{D}\left(\mathsf{F}_{\kappa}\right)^{**}}$,
so we can see $\mathsf{F}_{\kappa}$ as a subset of $\mathsf{\mathsf{D}\left(\mathsf{F}_{\kappa}\right)^{**}}$.
All that gives rise to the following definition.
\begin{defn}
Given a family $\left\{ \mathsf{E}_{\alpha}\right\} _{\alpha\in\kappa}$
of logics, we define its \textbf{tensor product }as the logic
\[
\left({\textstyle \bigotimes_{\alpha\in\kappa}}\mathsf{E}_{\alpha},\leq,\amalg,\wedge,*,\mathbf{0},\mathbf{1}\right),
\]
where
\begin{equation}
{\textstyle \bigotimes_{\alpha\in\kappa}}\mathsf{E}_{\alpha}=\mathsf{D}\left(\mathsf{F}_{\kappa}\right)^{**}=\left\{ \mathbb{A}\in\mathsf{D}\left(\mathsf{F}_{\kappa}\right):\mathbb{A}=\mathbb{A}^{**}\right\} =\left\{ \mathbb{A}^{*}:\mathbb{A}\in\mathsf{D}\left(\mathsf{F}_{\kappa}\right)\right\} \label{el-1}
\end{equation}
and with the same lattice structure as $\mathsf{D}\left(\mathsf{F}_{\kappa}\right)$,
except for the supremum, which is given by
\[
\coprod S=\left(\bigvee S\right)^{**},\quad\forall S\subseteq{\textstyle \bigotimes_{\alpha\in\kappa}}\mathsf{E}_{\alpha}.
\]
\end{defn}

\begin{rem*}
In the particular case in which $\mathsf{E}_{\alpha}=\mathsf{E}$
for all $\alpha$, it is clear that
\[
{\textstyle \bigotimes_{\alpha\in\kappa}}\mathsf{E}_{\alpha}=\mathsf{U}_{\kappa}\left(\mathsf{E}\right).
\]
\end{rem*}
Note that the elements of ${\textstyle \bigotimes_{\alpha\in\kappa}}\mathsf{E}_{\alpha}$
can be described as in \eqref{cbd}, with 
\[
\left({\textstyle \bigvee_{i\in I}}A^{i}\right)^{**}={\textstyle \bigvee_{i\in I}}A^{i},
\]
i.e. ${\textstyle \bigvee_{i\in I}}A^{i}={\textstyle \coprod_{i\in I}}A^{i}$.

\subsubsection{The embeddings $\mathfrak{i}_{\alpha}$}

Related to each $\alpha$, we have a natural embedding
\[
\mathfrak{i}_{\alpha}:a\in\mathsf{E}_{\alpha}\longmapsto\mathfrak{i}_{\alpha}\left(a\right)\in\mathsf{F}_{\kappa}
\]
with
\begin{equation}
\left(\mathfrak{i}_{\alpha}\left(a\right)\right)\left(\beta\right)=\begin{cases}
a, & \beta=\alpha,\\
\mathbf{1}, & \beta\neq\alpha,
\end{cases}\label{ia}
\end{equation}
which, \textit{via} the embedding $\varphi:\mathsf{F}_{\kappa}\rightarrow{\textstyle \bigotimes_{\beta\in\kappa}}\mathsf{E}_{\beta}$
(see \eqref{fial}), can be seen as a function from $\mathsf{E}_{\alpha}$
into ${\textstyle \bigotimes_{\beta\in\kappa}}\mathsf{E}_{\beta}$.
Below, by a \textit{logic embedding} we mean an injective homomorphism
of logics (i.e. of orthocomplemented complete lattices).
\begin{prop}
Each function $\mathfrak{i}_{\alpha}:\mathsf{E}_{\alpha}\rightarrow{\textstyle \bigotimes_{\beta\in\kappa}}\mathsf{E}_{\beta}$,
given by Eq. \eqref{ia}, is a logic embedding. In particular,
\[
\mathfrak{i}_{\alpha}\left(\bigcup S\right)=\coprod\mathfrak{i}_{\alpha}\left(S\right),\quad\mathfrak{i}_{\alpha}\left(\bigcap S\right)=\bigwedge\mathfrak{i}_{\alpha}\left(S\right),
\]
for all $S\subseteq\mathsf{E}_{\alpha}$, and 
\[
\mathfrak{i}_{\alpha}\left(\mathbf{0}\right)=\hat{\mathbf{0}},\quad\mathfrak{i}_{\alpha}\left(\mathbf{1}\right)=\hat{\mathbf{1}},\quad\mathfrak{i}_{\alpha}\left(a^{c}\right)=\left(\mathfrak{i}_{\alpha}\left(a\right)\right)^{*},
\]
for all $a\in\mathsf{E}_{\alpha}$.
\end{prop}

\begin{proof}
We already know that each $\mathfrak{i}_{\alpha}$ is injective. Given
a subset $\left\{ a_{i}\right\} _{i\in I}\subseteq\mathsf{E}_{\alpha}$,
since
\[
\left(\mathfrak{i}_{\alpha}\left({\textstyle \bigcap_{i\in I}}a_{i}\right)\right)\left(\beta\right)=\begin{cases}
{\textstyle \bigcap_{i\in I}}a_{i}, & \beta=\alpha,\\
\mathbf{1}, & \beta\neq\alpha,
\end{cases}=\begin{cases}
{\textstyle \bigcap_{i\in I}}a_{i}, & \beta=\alpha,\\
{\textstyle \bigcap_{i\in I}}\mathbf{1}, & \beta\neq\alpha,
\end{cases}
\]
and
\[
\left({\textstyle \bigwedge_{i\in I}}\mathfrak{i}_{\alpha}\left(a_{i}\right)\right)\left(\beta\right)={\textstyle \bigcap_{i\in I}}\left(\left(\mathfrak{i}_{\alpha}\left(a_{i}\right)\right)\left(\beta\right)\right)=\begin{cases}
{\textstyle \bigcap_{i\in I}}a_{i}, & \beta=\alpha,\\
{\textstyle \bigcap_{i\in I}}\mathbf{1}, & \beta\neq\alpha,
\end{cases}
\]
then $\mathfrak{i}_{\alpha}$ respects the meet operation. On the
other hand, given $a\in\mathsf{E}_{\alpha}$, we have that
\begin{equation}
\left(\mathfrak{i}_{\alpha}\left(a^{c}\right)\right)\left(\beta\right)=\begin{cases}
a^{c}, & \beta=\alpha,\\
\mathbf{1}, & \beta\neq\alpha.
\end{cases}\label{ia1}
\end{equation}
Also, according to Eq. \eqref{cf} (for $A=\mathfrak{i}_{\alpha}\left(a\right)$)
\[
\left(\mathfrak{i}_{\alpha}\left(a\right)\right)^{*}={\textstyle \bigvee_{\beta\in\kappa}}\mathfrak{i}_{\beta}\left(\left(\left(\mathfrak{i}_{\alpha}\left(a\right)\right)\left(\beta\right)\right)^{c}\right),
\]
and, since $\left(\left(\mathfrak{i}_{\alpha}\left(a\right)\right)\left(\beta\right)\right)^{c}=\mathbf{0}$
for $\beta\neq\alpha$ (see Eq. \eqref{ia}), then 
\[
{\textstyle \bigvee_{\beta\in\kappa}}\mathfrak{i}_{\beta}\left(\left(\left(\mathfrak{i}_{\alpha}\left(a\right)\right)\left(\beta\right)\right)^{c}\right)=\mathfrak{i}_{\alpha}\left(\left(\left(\mathfrak{i}_{\alpha}\left(a\right)\right)\left(\alpha\right)\right)^{c}\right)=\mathfrak{i}_{\alpha}\left(a^{c}\right),
\]
i.e.
\[
\left(\mathfrak{i}_{\alpha}\left(a\right)\right)^{*}=\mathfrak{i}_{\alpha}\left(a^{c}\right).
\]
All that (together with the De Morgan laws) shows $\mathfrak{i}_{\alpha}$
is a homomorphism of logics.
\end{proof}
\begin{rem*}
Note that, in terms of $\mathfrak{i}_{\alpha}$, Eq. \eqref{cf0}
translates to
\begin{equation}
A^{*}={\textstyle \bigvee_{\alpha\in\kappa}}\mathfrak{i}_{\alpha}\left(\left(A\left(\alpha\right)\right)^{c}\right).\label{cf}
\end{equation}
\end{rem*}
The next result will be important latter.
\begin{prop}
\label{ju} For every $A\in\mathsf{F}_{\kappa}$ and $J\subseteq\kappa$,
\[
{\textstyle \coprod_{\alpha\in J}}\mathfrak{i}_{\alpha}\left(A\left(\alpha\right)\right)={\textstyle \bigvee_{\alpha\in J}}\mathfrak{i}_{\alpha}\left(A\left(\alpha\right)\right).
\]
\end{prop}

\begin{proof}
We must show that
\[
\left({\textstyle \bigvee_{\alpha\in J}}\mathfrak{i}_{\alpha}\left(A\left(\alpha\right)\right)\right)^{**}={\textstyle \bigvee_{\alpha\in J}}\mathfrak{i}_{\alpha}\left(A\left(\alpha\right)\right).
\]
On the one hand, 
\begin{equation}
\left({\textstyle \bigvee_{\alpha\in J}}\mathfrak{i}_{\alpha}\left(A\left(\alpha\right)\right)\right)^{*}={\textstyle \bigwedge_{\alpha\in J}}\left(\mathfrak{i}_{\alpha}\left(A\left(\alpha\right)\right)\right)^{*}={\textstyle \bigcap_{\alpha\in J}}\mathfrak{i}_{\alpha}\left(\left(A\left(\alpha\right)\right)^{c}\right)\in\mathsf{F}_{\kappa}.\label{xa}
\end{equation}
On the other hand, given $B\in\mathsf{F}_{\kappa}$ (according to
\eqref{cf0}), 
\[
B^{*}={\textstyle \bigvee_{\alpha\in\kappa}}B_{\alpha},\quad\textrm{with}\quad B_{\alpha}\left(\beta\right)=\begin{cases}
\left(B\left(\alpha\right)\right)^{c}, & \beta=\alpha,\\
\mathbf{1}, & \beta\neq\alpha.
\end{cases}
\]
So, for 
\begin{equation}
B={\textstyle \bigcap_{\gamma\in J}}\mathfrak{i}_{\gamma}\left(\left(A\left(\gamma\right)\right)^{c}\right),\label{xb}
\end{equation}
since
\[
\left({\textstyle \bigcap_{\gamma\in J}}\mathfrak{i}_{\gamma}\left(\left(A\left(\gamma\right)\right)^{c}\right)\right)\left(\alpha\right)={\textstyle \bigcap_{\gamma\in J}}\left(\mathfrak{i}_{\gamma}\left(\left(A\left(\gamma\right)\right)^{c}\right)\right)\left(\alpha\right)
\]
and
\[
\left(\mathfrak{i}_{\gamma}\left(\left(A\left(\gamma\right)\right)^{c}\right)\right)\left(\alpha\right)=\begin{cases}
\left(A\left(\alpha\right)\right)^{c}, & \alpha=\gamma,\\
\mathbf{1}, & \alpha\neq\gamma,
\end{cases}
\]
then
\[
B\left(\alpha\right)=\begin{cases}
\left(A\left(\alpha\right)\right)^{c}, & \alpha\in J,\\
\mathbf{1}, & \alpha\notin J,
\end{cases}
\]
and consequently, 
\[
B_{\alpha}\left(\beta\right)=\begin{cases}
\left(B\left(\alpha\right)\right)^{c}=A\left(\alpha\right), & \beta=\alpha,\\
\mathbf{1}, & \beta\neq\alpha,
\end{cases},
\]
if $\alpha\in J$, and
\[
B_{\alpha}\left(\beta\right)=\begin{cases}
\left(\mathbf{1}\right)^{c}=\mathbf{0}, & \beta=\alpha,\\
\mathbf{1}, & \beta\neq\alpha,
\end{cases}=\hat{\mathbf{0}},
\]
if $\alpha\notin J$. This implies that
\[
B_{\alpha}=\begin{cases}
\mathfrak{i}_{\alpha}\left(A\left(\alpha\right)\right), & \alpha\in J,\\
\hat{\mathbf{0}}, & \alpha\notin J,
\end{cases}
\]
and then
\begin{equation}
B^{*}={\textstyle \bigvee_{\alpha\in\kappa}}B_{\alpha}={\textstyle \bigvee_{\alpha\in J}}\mathfrak{i}_{\alpha}\left(A\left(\alpha\right)\right).\label{xc}
\end{equation}
Summing up, from \eqref{xa}, \eqref{xb} and \eqref{xc},
\[
\left({\textstyle \bigvee_{\alpha\in J}}\mathfrak{i}_{\alpha}\left(A\left(\alpha\right)\right)\right)^{**}=B^{*}={\textstyle \bigvee_{\alpha\in J}}\mathfrak{i}_{\alpha}\left(A\left(\alpha\right)\right),
\]
what ends the proof. 
\end{proof}

\subsubsection{Distributive families}

Given a complete lattice $\mathsf{L}$ and a family of subsets $\mathfrak{S}\subseteq\mathcal{P}\left(\mathsf{L}\right)$,
consider its related set of choice functions
\[
\mathcal{C}\left(\mathfrak{S}\right)=\left\{ f:\mathfrak{S}\rightarrow\cup_{S\in\mathfrak{S}}S\quad/\quad f\left(S\right)\in S\right\} .
\]

\begin{defn}
\label{mj} We shall say the family $\mathfrak{S}$ is \textbf{meet-join
distributive} if, for every sub-family $\mathfrak{T}\subseteq\mathfrak{S}$,
\begin{equation}
\bigwedge\left\{ \bigvee S:S\in\mathfrak{T}\right\} =\bigvee\left\{ \bigwedge\mathsf{Im}f:f\in\mathcal{C}\left(\mathfrak{T}\right)\right\} .\label{mjcd}
\end{equation}
\end{defn}

Let us go back to the logic ${\textstyle \bigotimes_{\alpha\in\kappa}}\mathsf{E}_{\alpha}$
and, for each $A\in\mathsf{F}_{\kappa}$, consider the subset
\[
S_{A}\coloneqq\left\{ \mathfrak{i}_{\alpha}\left(A\left(\alpha\right)\right):\alpha\in\kappa\right\} \subseteq{\textstyle \bigotimes_{\alpha\in\kappa}}\mathsf{E}_{\alpha}.
\]

\begin{prop}
The family 
\begin{equation}
\mathfrak{S}\coloneqq\left\{ S_{A}:A\in\mathsf{F}_{\kappa}\right\} \subseteq\mathcal{P}\left({\textstyle \bigotimes_{\alpha\in\kappa}}\mathsf{E}_{\alpha}\right)\label{Si}
\end{equation}
is meet-join distributive.
\end{prop}

\begin{proof}
Every sub-family $\mathfrak{T}$ of $\mathfrak{S}$ is given by subsets
of the form
\[
\left\{ \mathfrak{i}_{\alpha}\left(A^{i}\left(\alpha\right)\right):\alpha\in\kappa\right\} ,\quad i\in I,
\]
where $I$ is an index set. In these terms, condition \eqref{mjcd}
reads
\[
{\textstyle \bigwedge_{i\in I}\coprod_{\alpha\in\kappa}}\mathfrak{i}_{\alpha}\left(A^{i}\left(\alpha\right)\right)={\textstyle \coprod_{f\in F}\bigcap_{i\in I}}\mathfrak{i}_{f\left(i\right)}\left(A^{i}\left(f\left(i\right)\right)\right),
\]
with $F=\left\{ f:I\rightarrow\kappa\right\} $. Since we can write
\[
{\textstyle \bigcap_{i\in I}}\mathfrak{i}_{f\left(i\right)}\left(A^{i}\left(f\left(i\right)\right)\right)={\textstyle \bigcap_{\alpha\in\kappa}\bigcap_{f\left(i\right)=\alpha}}\mathfrak{i}_{\alpha}\left(A^{i}\left(\alpha\right)\right)={\textstyle \bigcap_{\alpha\in\kappa}}\mathfrak{i}_{\alpha}\left({\textstyle \bigcap_{f\left(i\right)=\alpha}}A^{i}\left(\alpha\right)\right),
\]
we must show that
\begin{equation}
{\textstyle \bigwedge_{i\in I}\coprod_{\alpha\in\kappa}}\mathfrak{i}_{\alpha}\left(A^{i}\left(\alpha\right)\right)={\textstyle \coprod_{f\in F}\bigcap_{\alpha\in\kappa}}\mathfrak{i}_{\alpha}\left({\textstyle \bigcap_{f\left(i\right)=\alpha}}A^{i}\left(\alpha\right)\right).\label{ed}
\end{equation}
According to Proposition \ref{ju}, for all $i\in I$,
\begin{equation}
{\textstyle \coprod_{\alpha\in\kappa}}\mathfrak{i}_{\alpha}\left(A^{i}\left(\alpha\right)\right)={\textstyle \bigvee_{\alpha\in\kappa}}\mathfrak{i}_{\alpha}\left(A^{i}\left(\alpha\right)\right).\label{ea}
\end{equation}
Then, using the complete distributivity of $\mathsf{D}\left(\mathsf{F}_{\kappa}\right)$,
\begin{equation}
{\textstyle \bigwedge_{i\in I}\bigvee_{\alpha\in\kappa}}\mathfrak{i}_{\alpha}\left(A^{i}\left(\alpha\right)\right)={\textstyle \bigvee_{f\in F}\bigcap_{\alpha\in\kappa}}\mathfrak{i}_{\alpha}\left({\textstyle \bigcap_{f\left(i\right)=\alpha}}A^{i}\left(\alpha\right)\right).\label{eb}
\end{equation}
On the other hand, from \eqref{ea} we have that
\[
\left({\textstyle \bigwedge_{i\in I}\bigvee_{\alpha\in\kappa}}\mathfrak{i}_{\alpha}\left(A^{i}\left(\alpha\right)\right)\right)^{**}={\textstyle \bigwedge_{i\in I}\bigvee_{\alpha\in\kappa}}\mathfrak{i}_{\alpha}\left(\left(A^{i}\left(\alpha\right)\right)\right),
\]
so the same is true for ${\textstyle \bigvee_{f\in F}\bigcap_{\alpha\in\kappa}}\mathfrak{i}_{\alpha}\left({\textstyle \bigcap_{f\left(i\right)=\alpha}}A^{i}\left(\alpha\right)\right)$,
what implies that
\begin{equation}
{\textstyle \bigvee_{f\in F}\bigcap_{\alpha\in\kappa}}\mathfrak{i}_{\alpha}\left({\textstyle \bigcap_{f\left(i\right)=\alpha}}A^{i}\left(\alpha\right)\right)={\textstyle \coprod_{f\in F}\bigcap_{\alpha\in\kappa}}\mathfrak{i}_{\alpha}\left({\textstyle \bigcap_{f\left(i\right)=\alpha}}A^{i}\left(\alpha\right)\right).\label{ec}
\end{equation}
Combining \eqref{ea}, \eqref{eb} and \eqref{ec}, we have precisely
\eqref{ed}, as wanted.
\end{proof}
\begin{rem*}
At first appearance, the property fulfilled by the family $\mathfrak{S}$,
the meet-join distributivity, seems to be stronger than the \textit{ortho-distributivity
}condition (see Ref. \cite{mato}, p. 271). But if we look more closely,
we can see that the meet and joins, in each condition, run over different
set of indices.
\end{rem*}

\subsection{The universal property}

Related to the family $\left\{ \mathsf{E}_{\alpha}\right\} _{\alpha\in\kappa}$,
consider the pairs $\left(e,\mathsf{T}\right)$, where $\mathsf{T}$
is a logic $\left(\mathsf{T},\leq,\vee,\wedge,^{\bullet},\mathbf{0},\mathbf{1}\right)$
and $e=\left\{ e_{\alpha}\right\} _{\alpha\in\kappa}$ is a family
of logic embeddings $e_{\alpha}:\mathsf{E}_{\alpha}\rightarrow\mathsf{T}$,
i.e.

\[
e_{\alpha}\left(\bigcup S\right)=\bigvee e_{\alpha}\left(S\right),\quad e_{\alpha}\left(\bigcap S\right)=\bigwedge e_{\alpha}\left(S\right),
\]
for all $S\subseteq\mathsf{E}_{\alpha}$, and 
\[
e_{\alpha}\left(\mathbf{0}\right)=\mathbf{0},\quad e_{\alpha}\left(\mathbf{1}\right)=\mathbf{1},\quad e_{\alpha}\left(a^{c}\right)=\left(e_{\alpha}\left(a\right)\right)^{\bullet},
\]
for all $a\in\mathsf{E}_{\alpha}$. Assume in addition that the family
$\mathfrak{S}_{e}\subseteq\mathcal{P}\left(\mathsf{T}\right)$ given
by the sets
\[
S_{A}=\left\{ e_{\alpha}\left(A\left(\alpha\right)\right):\alpha\in\kappa\right\} ,\quad A\in\mathsf{F}_{\kappa},
\]
is meet-join distributive. This means that (see Definition \ref{mj}),
given a subfamily
\[
\left\{ e_{\alpha}\left(B^{i}\left(\alpha\right)\right):\alpha\in\kappa\right\} ,\quad i\in I,
\]
being $I$ an index set, we have the identity (compare to \eqref{ed})
\begin{equation}
{\textstyle \bigwedge_{i\in I}\bigvee_{\alpha\in\kappa}}e_{\alpha}\left(B^{i}\left(\alpha\right)\right)={\textstyle \bigvee_{f\in F}\bigwedge_{\alpha\in\kappa}}e_{\alpha}\left({\textstyle \bigcap_{f\left(i\right)=\alpha}}B^{i}\left(\alpha\right)\right).\label{edb}
\end{equation}

\bigskip{}

As an example of above pairs we have $\left(\mathfrak{i},{\textstyle \bigotimes_{\alpha\in\kappa}}\mathsf{E}_{\alpha}\right)$,
where $\mathfrak{i}$ denotes the set of logic embeddings $\mathfrak{i}_{\alpha}$
defined in the previous section (see Eq. \eqref{ia}). The related
meet-join distributive family is $\mathfrak{S}_{\mathfrak{i}}\coloneqq\mathfrak{S}$
(see \eqref{Si}). For such a pair, we have the following universal
property.
\begin{thm}
There exists a unique homomorphism of logics $t:{\textstyle \bigotimes_{\alpha\in\kappa}}\mathsf{E}_{\alpha}\rightarrow\mathsf{T}$
such that 
\begin{equation}
t\left(\mathfrak{i}_{\alpha}\left(a\right)\right)=e_{\alpha}\left(a\right),\quad\forall a\in\mathsf{E}_{\alpha},\quad\forall\alpha\in\kappa.\label{tia}
\end{equation}
\end{thm}

\begin{proof}
Consider the function $\tilde{t}:\mathsf{D}\left(\mathsf{F}_{\kappa}\right)\rightarrow\mathsf{T}$
such that 
\begin{equation}
\tilde{t}\left({\textstyle \bigvee_{i\in I}}A^{i}\right)={\textstyle \bigvee_{i\in I}}{\textstyle \bigwedge_{\alpha\in\kappa}}e_{\alpha}\left(A^{i}\left(\alpha\right)\right).\label{dt}
\end{equation}
Let us see $\tilde{t}$ is well-defined. If ${\textstyle \bigvee_{i\in I}}A^{i}={\textstyle \bigvee_{j\in J}}B^{j}$,
this means that for all $i\in I$ there exists $j\left(i\right)\in J$
such that $A^{i}\subseteq B^{j\left(i\right)}$, i.e. 
\[
A^{i}\left(\alpha\right)\subseteq B^{j\left(i\right)}\left(\alpha\right),\quad\forall\alpha\in\kappa,
\]
and vice versa (changing $A$'s by $B$'s). So, 
\[
\begin{array}{lll}
\tilde{t}\left({\textstyle \bigvee_{i\in I}}A^{i}\right) & = & {\textstyle \bigvee_{i\in I}}{\textstyle \bigwedge_{\alpha\in\kappa}}e_{\alpha}\left(A^{i}\left(\alpha\right)\right)\leq{\textstyle \bigvee_{i\in I}}{\textstyle \bigwedge_{\alpha\in\kappa}}e_{\alpha}\left(B^{j\left(i\right)}\left(\alpha\right)\right)\\
\\
 & \leq & {\textstyle \bigvee_{j\in J}}{\textstyle \bigwedge_{\alpha\in\kappa}}e_{\alpha}\left(B^{j}\left(\alpha\right)\right)=\tilde{t}\left({\textstyle \bigvee_{j\in J}}B^{j}\right),
\end{array}
\]
and the other inequality follows by changing $A$'s by $B$'s. 

Note that, if $A\in\mathsf{F}_{\kappa}$, then we can write $A=\bigvee_{i\in I}A^{i}$
with $I=\left\{ s\right\} $ and $A^{s}=A$. Thus, according to \eqref{dt},
\begin{equation}
\tilde{t}\left(A\right)={\textstyle \bigwedge_{\alpha\in\kappa}}e_{\alpha}\left(A\left(\alpha\right)\right).\label{dt1}
\end{equation}
In particular, given $a\in\mathsf{E}_{\beta}$,
\[
\tilde{t}\left(\mathfrak{i}_{\beta}\left(a\right)\right)={\textstyle \bigwedge_{\alpha\in\kappa}}e_{\alpha}\left(\left(\mathfrak{i}_{\beta}\left(a\right)\right)\left(\alpha\right)\right).
\]
Since $\left(\mathfrak{i}_{\beta}\left(a\right)\right)\left(\alpha\right)=\mathbf{1}$
when $\alpha\neq\beta$ (see \eqref{ia}), then 
\begin{equation}
\tilde{t}\left(\mathfrak{i}_{\beta}\left(a\right)\right)=e_{\beta}\left(\left(\mathfrak{i}_{\beta}\left(a\right)\right)\left(\beta\right)\right)=e_{\beta}\left(a\right),\quad\forall a\in\mathsf{E}_{\beta},\quad\forall\beta\in\kappa,\label{tia2}
\end{equation}
because $\left(\mathfrak{i}_{\beta}\left(a\right)\right)\left(\beta\right)=a$
(see \eqref{ia} again).

From \eqref{dt} and \eqref{dt1}, we have that
\begin{equation}
\tilde{t}\left({\textstyle \bigvee_{i\in I}}A^{i}\right)={\textstyle \bigvee_{i\in I}}\tilde{t}\left(A^{i}\right),\quad\forall\;{\textstyle \bigvee_{i\in I}}A^{i}\in\mathsf{D}\left(\mathsf{F}_{\kappa}\right),\label{tas}
\end{equation}
and given a subset 
\[
\left\{ \mathbb{A}_{j}:j\in J\right\} \subseteq\mathsf{D}\left(\mathsf{F}_{\kappa}\right),
\]
with $\mathbb{A}_{j}={\textstyle \bigvee_{i\in I_{j}}}A_{j}^{i}$,
\begin{equation}
\begin{array}{lll}
\tilde{t}\left({\textstyle \bigvee_{j\in J}}\mathbb{A}_{j}\right) & = & \tilde{t}\left({\textstyle \bigvee_{j\in J}}{\textstyle \bigvee_{i\in I_{j}}}A_{j}^{i}\right)={\textstyle \bigvee_{j\in J}}{\textstyle \bigvee_{i\in I_{j}}}\tilde{t}\left(A_{j}^{i}\right)\\
\\
 & = & {\textstyle \bigvee_{j\in J}}\tilde{t}\left({\textstyle \bigvee_{i\in I_{j}}}A_{j}^{i}\right)={\textstyle \bigvee_{j\in J}}\tilde{t}\left(\mathbb{A}_{j}\right).
\end{array}\label{ttu}
\end{equation}
Also, from \eqref{fc}, \eqref{dt1} and \eqref{tas},
\[
\tilde{t}\left(\left({\textstyle \bigvee_{i\in I}}A^{i}\right)^{*}\right)=\tilde{t}\left({\textstyle \bigvee_{f\in F}}\mathbb{A}_{f,I}\right)={\textstyle \bigvee_{f\in F}}\tilde{t}\left(\mathbb{A}_{f,I}\right)
\]
and
\[
\tilde{t}\left(\mathbb{A}_{f,I}\right)={\textstyle \bigwedge_{\alpha\in\kappa}}e_{\alpha}\left(\mathbb{A}_{f,I}\left(\alpha\right)\right)={\textstyle \bigwedge_{\alpha\in\kappa}}e_{\alpha}\left({\textstyle \bigcap_{f\left(i\right)=\alpha}}\left(A^{i}\left(\alpha\right)\right)^{c}\right).
\]
On the other hand, the meet-join distributivity property for $\left(e,\mathsf{T}\right)$
ensures that\footnote{See Eq. \eqref{edb} for $B^{i}\left(\alpha\right)=\left(A^{i}\left(\alpha\right)\right)^{c}$.}
\[
{\textstyle \bigvee_{f\in F}}{\textstyle \bigwedge_{\alpha\in\kappa}}e_{\alpha}\left({\textstyle \bigcap_{f\left(i\right)=\alpha}}\left(A^{i}\left(\alpha\right)\right)^{c}\right)={\textstyle \bigwedge_{i\in I}\bigvee_{\alpha\in\kappa}}e_{\alpha}\left(\left(A^{i}\left(\alpha\right)\right)^{c}\right)=\left({\textstyle \bigvee_{i\in I}\bigwedge_{\alpha\in\kappa}}e_{\alpha}\left(A^{i}\left(\alpha\right)\right)\right)^{\bullet},
\]
so
\begin{equation}
\tilde{t}\left(\left({\textstyle \bigvee_{i\in I}}A^{i}\right)^{*}\right)=\left({\textstyle \bigvee_{i\in I}\bigwedge_{\alpha\in\kappa}}e_{\alpha}\left(A^{i}\left(\alpha\right)\right)\right)^{\bullet}=\left(\tilde{t}\left({\textstyle \bigvee_{i\in I}}A^{i}\right)\right)^{\bullet}.\label{tst}
\end{equation}

Now, define $t:{\textstyle \bigotimes_{\alpha\in\kappa}}\mathsf{E}_{\alpha}\rightarrow\mathsf{T}$
as the restriction of $\tilde{t}$. Since $\mathsf{F}_{\kappa}\subseteq{\textstyle \bigotimes_{\alpha\in\kappa}}\mathsf{E}_{\alpha}$,
it is clear from \eqref{tia2} that $t$ satisfies \eqref{tia}. Also,
given ${\textstyle \bigvee_{i\in I}}A^{i}\in{\textstyle \bigotimes_{\alpha\in\kappa}}\mathsf{E}_{\alpha}$,
since $\left({\textstyle \bigvee_{i\in I}}A^{i}\right)^{*}\in{\textstyle \bigotimes_{\alpha\in\kappa}}\mathsf{E}_{\alpha}$
too, then \eqref{tst} implies that
\[
t\left(\left({\textstyle \bigvee_{i\in I}}A^{i}\right)^{*}\right)=\left(t\left({\textstyle \bigvee_{i\in I}}A^{i}\right)\right)^{\bullet},
\]
i.e. $t$ respect the orthocomplementation. Finally, given an arbitrary
subset 
\[
\left\{ \mathbb{A}_{j}:j\in J\right\} \subseteq{\textstyle \bigotimes_{\alpha\in\kappa}}\mathsf{E}_{\alpha},
\]
from \eqref{ttu} and \eqref{tst} we have that
\[
\begin{array}{lll}
t\left({\textstyle \coprod_{j\in J}}\mathbb{A}_{j}\right) & = & t\left(\left({\textstyle \bigvee_{j\in J}}\mathbb{A}_{j}\right)^{**}\right)=\tilde{t}\left(\left({\textstyle \bigvee_{j\in J}}\mathbb{A}_{j}\right)^{**}\right)=\left(\tilde{t}\left({\textstyle \bigvee_{j\in J}}\mathbb{A}_{j}\right)\right)^{\bullet\bullet}\\
\\
 & = & \tilde{t}\left({\textstyle \bigvee_{j\in J}}\mathbb{A}_{j}\right)={\textstyle \bigvee_{j\in J}}\tilde{t}\left(\mathbb{A}_{j}\right)={\textstyle \bigvee_{j\in J}}t\left(\mathbb{A}_{j}\right).
\end{array}
\]
The fact that $t$ respects the meets follows from the De Morgan laws.
Summing up, the function $t:{\textstyle \bigotimes_{\alpha\in\kappa}}\mathsf{E}_{\alpha}\rightarrow\mathsf{T}$
such that
\begin{equation}
t\left({\textstyle \bigvee_{i\in I}}A^{i}\right)={\textstyle \bigvee_{i\in I}}{\textstyle \bigwedge_{\alpha\in\kappa}}e_{\alpha}\left(A^{i}\left(\alpha\right)\right)\label{dt2}
\end{equation}
 is a homomorphism of logics that satisfies \eqref{tia}. It rests
to see that it is the only one.

\bigskip{}

Note first that any element $A\in\mathsf{F}_{\kappa}$ satisfies
\begin{equation}
A={\textstyle \bigwedge_{\alpha\in\kappa}}\mathfrak{i}_{\alpha}\left(A\left(\alpha\right)\right).\label{aA}
\end{equation}
In fact, ${\textstyle \bigwedge_{\alpha\in\kappa}}\mathfrak{i}_{\alpha}\left(A\left(\alpha\right)\right)={\textstyle \bigcap_{\alpha\in\kappa}}\mathfrak{i}_{\alpha}\left(A\left(\alpha\right)\right)\in\mathsf{F}_{\kappa}$
and
\[
\left({\textstyle \bigcap_{\alpha\in\kappa}}\mathfrak{i}_{\alpha}\left(A\left(\alpha\right)\right)\right)\left(\beta\right)={\textstyle \bigcap_{\alpha\in\kappa}}\left(\mathfrak{i}_{\alpha}\left(A\left(\alpha\right)\right)\left(\beta\right)\right)=\left(\mathfrak{i}_{\beta}\left(A\left(\beta\right)\right)\left(\beta\right)\right)=A\left(\beta\right).
\]
Suppose that $t':{\textstyle \bigotimes_{\alpha\in\kappa}}\mathsf{E}_{\alpha}\rightarrow\mathsf{T}$
is a homomorphism of logics satisfying \eqref{tia}. If ${\textstyle \bigvee_{i\in I}}A^{i}\in{\textstyle \bigotimes_{\alpha\in\kappa}}\mathsf{E}_{\alpha}$,
then
\[
{\textstyle \bigvee_{i\in I}}A^{i}={\textstyle \coprod_{i\in I}}A^{i}={\textstyle \coprod_{i\in I}}{\textstyle \bigwedge_{\alpha\in\kappa}}\mathfrak{i}_{\alpha}\left(A^{i}\left(\alpha\right)\right)
\]
and (recall \eqref{dt})
\[
\begin{array}{lll}
t'\left({\textstyle \bigvee_{i\in I}}A^{i}\right) & = & t'\left({\textstyle \coprod_{i\in I}}A^{i}\right)=t'\left({\textstyle \coprod_{i\in I}}{\textstyle \bigwedge_{\alpha\in\kappa}}\mathfrak{i}_{\alpha}\left(A^{i}\left(\alpha\right)\right)\right)={\textstyle \bigvee_{i\in I}}{\textstyle \bigwedge_{\alpha\in\kappa}}t'\left(\mathfrak{i}_{\alpha}\left(A^{i}\left(\alpha\right)\right)\right)\\
\\
 & = & {\textstyle \bigvee_{i\in I}}{\textstyle \bigwedge_{\alpha\in\kappa}}e_{\alpha}\left(A^{i}\left(\alpha\right)\right)=t\left({\textstyle \bigvee_{i\in I}}A^{i}\right),
\end{array}
\]
what ensures that $t'=t$, as wanted. 
\end{proof}

\section{Future work}

\label{fw} In a seminal work of R.T. Cox \cite{cox} (see also \cite{cox2}
and \cite{jaynes}), an axiomatic foundation of probability in the
subjectivist interpretation (i.e. interpreted as a \textit{reasonable
expectation}) is presented. The work is made in the special context
of classical experiments, i.e. in the case of event spaces $\mathsf{E}$
given by Boolean algebras. In that work, Cox concludes that his axioms
(the \textit{Cox's axioms} from now on) give rise exactly to the usual
notion of probability, as defined by Kolmogorov \cite{kol}. But,
as noted in Ref. \cite{hal}, Cox's axioms are not enough to arrive
at such a conclusion. An additional condition is needed: the \textit{density
assumption }(see \cite{paris} and \cite{vanh}). Unfortunately, the
density assumption only makes sense when the event space $\mathsf{E}$
is infinite, excluding in this way the simple experiment of tossing
a coin. To solve this problem, in the unpublished pre-print \cite{tere},
the authors propose to study not only the reasonable expectation for
the event space $\mathsf{E}$ of the original experiment, but also
for the event space of the indefinitely repeated experiment. Since
the latter is always infinite, the density assumption has a greater
chance of being satisfied. In a work in progress we are considering
the idea of \cite{tere}, but in the context of general experiments,
i.e. experiments whose event space $\mathsf{E}$ is a general logic.
More concretely, we are studying a possible analogue of the Cox's
axioms in this more general situation, for which the universal logic
$\mathsf{U}_{\mathbb{N}}\left(\mathsf{E}\right)$ is playing a central
role.

\section*{Appendix: experiments, results and events}

Let us call \textit{procedure} to any set of duly prescribed actions
(e.g. those that are usually realized in a laboratory), and \textit{trial}
(or \textit{individual experiment}) to any ordered pair of procedures
$\left(p,m\right)$. Call \textit{preparation} (or \textit{stimulus})
to the first component $p$ and \textit{measurement} (or \textit{response})
to the second one $m$. Given two sets of procedures $\mathbb{P}$
and $\mathbb{M}$, we shall call \textbf{experiment} to the set of
trials $\mathbb{P}\times\mathbb{M}$, and we will say that $\mathbb{P}$
is the set of preparations of the experiment and $\mathbb{M}$ is
its set of measurements. In practice, after running a preparation
$p\in\mathbb{P}$ by the experimenter, it can follow only a certain
subclass of measurements $m\in\mathbb{M}$. This defines a relation
$\mathfrak{R}\subseteq\mathbb{P}\times\mathbb{M}$: the subset of
\textit{empirical trials}. Schematically: 
\begin{center}
\begin{tabular}{|c|c|c|c|c|c|c|}
\hline 
$\mathbb{P}$ & $\times$ & $\mathbb{M}$ & $=$ & $\mathbb{P}\times\mathbb{M}$ & $\supseteq$ & $\mathfrak{R}$\tabularnewline
\hline 
$\textrm{preparations}$ &  & $\textrm{measurements}$ &  & $\textrm{experiment}$ &  & $\textrm{empirical relation }$\tabularnewline
\hline 
\end{tabular}
\par\end{center}

Let us assume that $\mathbb{M}$ is partitioned into a disjoint union,
i.e. $\mathbb{M}=\bigvee_{\alpha}\mathbb{O}_{\alpha}$, where each
subset $\mathbb{O}_{\alpha}\subseteq\mathbb{M}$, which we shall call
\textit{observable}, is in bijection with a subset $I_{\alpha}\subseteq\mathbb{R}$,
whose elements we shall call \textbf{results} of $\mathbb{O}_{\alpha}$. 
\begin{rem*}
When the experiment corresponds to a classical (resp. quantum) mechanical
system with phase space $P$ (resp. with Hilbert space $\mathcal{H}$),
each observable $\mathbb{O}_{\alpha}$ is represented by a real-valued
function $f_{\alpha}$ on $P$ (resp. by a self-adjoint operator $O_{\alpha}$
on $\mathcal{H}$), and the subset $I_{\alpha}$ is the range of $f_{\alpha}$
(resp. the spectrum of $O_{\alpha}$).
\end{rem*}
We can assume (just for the sake of simplification) that all the procedures
inside $\mathbb{O}_{\alpha}$ are defined by a series of actions (prescribed
by a precise algorithm) that ends in writing a number $r\in I_{\alpha}$
in a notebook, and the only difference among such series of actions
is precisely the number that we must write: the result $r$. In other
words, the procedures in $\mathbb{O}_{\alpha}$ can be described by
pairs $\left(\alpha,r\right)$, with $r\in I_{\alpha}$ and $\alpha$
being the (sub)procedure that they all have in common. If the experimenter
runs $\left(\alpha,r\right)$ after running $p$, i.e. if we have
the empirical trial $\left(p,\left(\alpha,r\right)\right)\in\mathfrak{R}$,
we shall say that ``he/she obtained the result $r$ in the measurement
of the observable $\mathbb{O}_{\alpha}$ under the preparation $p$.''

In this scenario, fixing a preparation $p$, some of the possible
propositions about the results of the experiment are those of the
form: ``in the measurement of the observable $\mathbb{O}_{\alpha}$,
the result belongs to the subset $U\subseteq I_{\alpha}$''. We can
denote $\left(\alpha,U\right)$ to such propositions. Thus, a possible
event space $\mathsf{E}$ for the experiment $\mathbb{P}\times\mathbb{M}$
is the one given by the propositions $\left(\alpha,U\right)$ together
with those obtained from them by applying the logical connectives
AND, OR and NOT.

\section*{Acknowledgments}

The author thanks CONICET for its financial support. He also thanks
Martín Onetto, Diego Morcelle and José Luis Castiglione for their
useful comments and the fruitful discussions.

\end{document}